\documentclass[a4paper,11pt,twoside]{amsart}

\usepackage{amsfonts}
\usepackage{amsmath}
\usepackage{amsthm}
\usepackage{amssymb}
\usepackage[all]{xy}
\usepackage{hyperref}

\newtheorem{thm}{Theorem}[section]
\newtheorem{prop}[thm]{Proposition}
\newtheorem{lem}[thm]{Lemma}
\newtheorem{cor}[thm]{Corollary}

\numberwithin{equation}{section}

\theoremstyle{definition}
\newtheorem{definition}[thm]{Definition}
\newtheorem{remark}[thm]{Remark}
\newtheorem{ex}[thm]{Example}

\DeclareMathOperator{\im}{Im} 
\DeclareMathOperator{\cok}{coker} 

\newcommand{\iso}{\cong}

\newcommand{\farg}{-} 
\newcommand{\id}{\mathrm{id}}
\newcommand{\comp}{\circ} 
\newcommand{\mor}[1]{\xrightarrow{#1}}
\newcommand{\isomor}{\mor{\sim}} 
\newcommand{\rest}[1]{|_{#1}} 
\newcommand{\K}{\Bbbk} 
\newcommand{\ring}[1]{\mathbb{#1}}
\newcommand{\Z}{\ring{Z}} 
\newcommand{\dg}{\mathrm{dg}}
\newcommand{\D}[1][]{\mathrm{D}^{#1}} 
\newcommand{\Db}{\D[b]} 
\newcommand{\cat}[1]{{\mathbf{#1}}} 
\newcommand{\opp}{^{\circ}} 
\newcommand{\fun}[1]{\mathsf{#1}} 
\newcommand{\dgcat}[1]{\cat{#1}} 
\newcommand{\dgfun}[1]{\fun{#1}} 
\newcommand{\dgker}[1]{\mathcal{#1}} 
\newcommand{\FM}[2][]{\Phi^{#1}_{#2}} 
\newcommand{\dfun}[1]{\Phi_{#1}} 
\newcommand{\ext}[1]{\widehat{#1}} 
\newcommand{\res}[1]{\widetilde{#1}} 
\newcommand{\dgFM}[1]{\ext{\dfun{#1}}} 
\newcommand{\dgm}[1]{\mathrm{dgMod}(#1)} 
\newcommand{\dgh}[1]{\mathrm{dgMod_{hp}}(#1)} 
\newcommand{\Pe}[1]{\mathrm{Perf}(#1)} 
\newcommand{\hproj}[1]{\mathrm{h\text{-}proj}(#1)} 
\newcommand{\hpdg}[1]{{#1}^\mathrm{hp}}
\newcommand{\essim}[1]{\overline{#1}} 
\newcommand{\dgMor}[1]{\mathrm{Mor}(#1)} 
\newcommand{\dgFun}{\underline{Hom}} 
\newcommand{\dgCat}{\cat{dgCat}} 
\newcommand{\Hqe}{\cat{Hqe}} 
\newcommand{\hpdgCat}{\cat{hp\text{-}dgCat}} 
\newcommand{\hpHqe}{\cat{hp\text{-}Hqe}} 
\newcommand{\hqe}[1]{[#1]} 
\newcommand{\hq}[1]{[#1]} 
\newcommand{\Ind}[1]{\mathrm{Ind}_{#1}} 
\newcommand{\Res}[1]{\mathrm{Res}_{#1}} 
\newcommand{\src}[1]{\dgfun{S}_{#1}} 
\newcommand{\tar}[1]{\dgfun{T}_{#1}} 
\newcommand{\inc}[1]{\dgfun{I}_{#1}} 
\newcommand{\gpob}[1]{#1^I} 
\newcommand{\pob}[1]{P(#1)} 
\newcommand{\Cdg}[1][\K]{\mathrm{C}_{\dg}(#1)} 
\newcommand{\Yon}[1]{\fun{Y}_{#1}} 
\newcommand{\ho}[1]{[#1]_{\mathrm{iso}}} 
\newcommand{\dgid}[1]{\Delta_{#1}} 
\newcommand{\monomap}[1][]{\Sigma_{#1}} 
\newcommand{\isomap}[1][]{\Lambda_{#1}} 
\newcommand{\Hom}{\mathrm{Hom}}
\newcommand{\Iso}{\mathrm{Iso}}

\newcommand{\R}{\mathbb{R}} 
\newcommand{\Ob}{\mathrm{Ob}}

\newcommand{\ZZ}{\mathbb{Z}}

\newcommand{\tto}{\longrightarrow}

\newcommand{\rqr}[1]{\mathrm{h\text{-}proj}(#1)^{\mathrm{rqr}}} 
\newcommand{\IHom}{\R\dgFun}
\newcommand{\rd}{\mathbf{R}} 
\newcommand{\ld}{\mathbf{L}} 
\newcommand{\lotimes}{\otimes^{\ld}} 

\newcommand{\dgA}{\dgcat{A}} 
\newcommand{\dgB}{\dgcat{B}} 
\newcommand{\dgC}{\dgcat{C}} 
\newcommand{\dgD}{\dgcat{D}} 
\newcommand{\dgF}{\dgfun{F}} 
\newcommand{\dgG}{\dgfun{G}} 
\newcommand{\dgH}{\dgfun{H}} 
\newcommand{\dgI}{\dgfun{I}} 
\newcommand{\dgJ}{\dgfun{J}} 
\newcommand{\dgQ}{\dgfun{Q}} 
\newcommand{\ke}{\dgker{E}} 

\newcommand{\kkd}{D}
\newcommand{\kke}{E}
\newcommand{\kkf}{F}
\newcommand{\kkg}{G}
\newcommand{\kki}{I}

\begin{document}

	\title[Internal Homs via extensions of dg functors]{Internal Homs via extensions of dg functors}

	\author{Alberto Canonaco and Paolo Stellari}

	\address{A.C.: Dipartimento di Matematica ``F. Casorati'', Universit{\`a}
	degli Studi di Pavia, Via Ferrata 1, 27100 Pavia, Italy}
	\email{alberto.canonaco@unipv.it}

	\address{P.S.: Dipartimento di Matematica ``F.
	Enriques'', Universit{\`a} degli Studi di Milano, Via Cesare Saldini
	50, 20133 Milano, Italy}
	\email{paolo.stellari@unimi.it}
    \urladdr{\url{https://sites.unimi.it/stellari}}
	
	\thanks{A.~C.~ was partially supported by the national research project
	  ``Moduli, strutture geometriche e loro applicazioni'' (PRIN 2009).
	P.~S.~ is partially supported by the grants FIRB 2012 ``Moduli Spaces and Their Applications'' and
	the national research project ``Geometria delle Variet\`a Proiettive'' (PRIN 2010-11).}



\begin{abstract}
We provide a simple proof of the existence of internal Homs in the localization of the category of dg categories with respect to all quasi-equivalences and of some of their main properties such as the so-called derived Morita theory. This was originally proved in a seminal paper by To\"{e}n.
\end{abstract}

\maketitle

\section{Introduction}\label{sec:intro}

The problem of characterizing exact functors between
triangulated categories is certainly one of the major open questions
in the theory of triangulated categories. As
soon as we deal with triangulated categories which are the bounded
derived categories of coherent sheaves on smooth projective varieties,
this challenge in the vague form above gets neater. More precisely, if $X_1$ and $X_2$ are smooth
projective schemes and we denote by $\Db(X_i)$ the bounded derived category of
coherent sheaves on $X_i$, then one would expect that all exact functors
$\fun{F}:\Db(X_1)\tto\Db(X_2)$
are of \emph{Fourier--Mukai type} (see \cite{BLL, Or1}). This means that there should exist
$\ke\in\Db(X_1\times X_2)$ and an isomorphism of
exact functors $\fun{F}\iso\FM{\ke}$, where, denoting by $p_i:X_1\times
X_2\to X_i$ the natural projections, $\FM{\ke}:\Db(X_1)\to\Db(X_2)$ is
the exact functor defined by
$\FM{\ke}:=\rd(p_2)_*(\ke\lotimes p_1^*(-))$.
Unfortunately, only partial results confirm the expectation above
(see \cite{CS1} for a survey about this).

If one changes perspective slightly and moves to higher categorical structures, the situation becomes amazingly
beautiful. More precisely, one looks at the localization $\Hqe$ of the category $\dgCat$ of (small) dg
categories over a commutative ring $\K$ with respect to all
quasi-equivalences. Then one can take dg enhancements $\dgD_1$ and $\dgD_2$ of $\Db(X_1)$ and $\Db(X_2)$ respectively. It turns out (see \cite{BLL}) that all Fourier--Mukai functors at the triangulated level lift to morphisms between $\dgD_1$ and $\dgD_2$ in $\Hqe$. More surprisingly, all morphisms in $\Hqe$ between these dg categories are of this type.
This was observed in the seminal paper \cite{To} and comes as a corollary of
a very general and elegant result about the existence of internal Homs in $\Hqe$. The statement can be formulated as follows.

\begin{thm}\label{thm:Toen}{\bf (To\"{e}n, \cite{To})}
Let $\dgA$, $\dgB$ and $\dgC$ be three dg categories over a commutative ring $\K$. Then there exists a natural bijection
\begin{equation}\label{eqn:firstpart}
\xymatrix{
\hqe{\dgA,\dgC}\ar@{<->}[rr]^-{1:1}&&\Iso(H^0(\rqr{\dgA\opp\lotimes\dgC})).
}
\end{equation}
Moreover, the dg category $\IHom(\dgB,\dgC):=\rqr{\dgB\opp\lotimes\dgC}$ yields a natural bijection
\begin{equation}\label{eqn:ultima}
\xymatrix{
\hqe{\dgA\lotimes\dgB,\dgC}\ar@{<->}[rr]^-{1:1}&&\hqe{\dgA,\IHom(\dgB,\dgC)}
}
\end{equation}
proving that the symmetric monoidal category $\Hqe$ is closed.
\end{thm}

Here $\hqe{\farg,\farg}$ denotes the set of morphisms in $\Hqe$,
while $\rqr{\farg}$ denotes right quasi-representable h-projective dg
modules, which will be precisely described in Section
\ref{subsect:ext}. In the original version \cite{To} $\rqr{\farg}$ is
replaced by the dg category of right quasi-representable cofibrant dg
modules, which is actually quasi-equivalent to $\rqr{\farg}$.
Recall that the monoidal
structure provided by the derived tensor product $\farg\lotimes\farg$ is said to be \emph{closed}
if, for every $\dgB,\dgC\in\Hqe$, there exists
$\IHom(\dgB,\dgC)\in\Hqe$, such that the functor
$\dgA\mapsto\hqe{\dgA\lotimes\dgB,\dgC}$ is isomorphic to
$\dgA\mapsto\hqe{\dgA,\IHom(\dgB,\dgC)}$.

Roughly speaking, Theorem \ref{thm:Toen} asserts that all dg (quasi-)functors are of Fourier--Mukai type. The reason is that, if one looks carefully at the proof of the bijection \eqref{eqn:firstpart}, one sees that all morphisms in $\Hqe$ are essentially provided by the tensorization by dg bimodules, mimicking the definition of Fourier--Mukai functors given above in the triangulated setting.
It is worth pointing out that the first part of Theorem \ref{thm:Toen} comes in \cite{To} as a corollary of a much more general result involving substantially the simplicial structure on $\dgCat$, seen as a model category (see \cite[Thm.\ 4.2]{To}).

The purpose of this paper is to provide a simple proof of Theorem \ref{thm:Toen}
essentially based on the circle of ideas emerging from \cite{Dr} and \cite{Ke1}.

\subsection*{Comparing To\"{e}n's approach and ours} Following \cite{To}, one uses the model category structure on $\dgCat$ (see \cite{TaDg}) in such a way that any morphism in $\Hqe$ can be represented by a
`canonical' roof by means of the cofibrant replacements. At this point, the description
of the morphisms between two dg categories in $\Hqe$ can be essentially carried out
assuming that we are working with actual dg functors. In this way, To\"{e}n proves that, for two dg categories $\dgA$ and $\dgB$,
there is a bijection between $\hqe{\dgA,\dgB}$ and the isomorphism classes
of the homotopy category of right quasi-representable (fibrant and) cofibrant
$\dgA\opp\otimes\dgB$-dg modules (see Corollary 4.10 of \cite{To}). This is essentially the first part of Theorem \ref{thm:Toen}.

The existence of internal Homs follows from a characterization of the model category of dg functors between dg categories and a comparison between
this and the presentation above
(\cite[Thm.\ 6.1]{To}). As an application of the existence of internal
Homs, To\"{e}n deduces in \cite[Thm.\ 7.2]{To} a restriction theorem asserting in particular that, given two dg
categories $\dgA$ and $\dgB$, the Yoneda embedding of $\dgA$ into the dg
category $\mathrm{Int}(\dgA)$ of cofibrant $\dgA$-dg modules yields
a quasi-equivalence between the continuous internal Hom
$\IHom_c(\mathrm{Int}(\dgA),\mathrm{Int}(\dgB))$ and
$\IHom(\dgA,\mathrm{Int}(\dgB))$ (see Sections \ref{subsec:tensors} and \ref{sec:Morita} for the precise definitions), which goes under the name of \emph{derived Morita theory}.
This is a very interesting result in itself with various geometric
applications as explained in \cite[Sect.\ 8]{To}. One should keep in
mind that, for a dg category $\dgA$, the dg category
$\mathrm{Int}(\dgA)$ is contained in the dg category $\hproj{\dgA}$ of
h-projective $\dgA$-dg modules and the inclusion is a
quasi-equivalence (see, for example, \cite[Lemma 2.6]{LS} and the discussion in Section 3.2 of \cite{Ke} for this standard fact).

In a sense, our argument starts from this sort of ending point in
\cite{To}. Indeed, we prove directly a weaker form of this restriction result (see
Proposition \ref{prop:Morita} and Corollary \ref{cor:bij}) which provides a nice description of some morphisms in $\Hqe$ in terms of isomorphism classes of special dg bimodules (see Proposition \ref{prop:exuniq}).
The existence of internal
Homs and the proof of Theorem \ref{thm:Toen} follow then from a simple and purely
formal argument explained in Section \ref{sec:Morita}. All of this is achieved using the notion of extension of dg functors which is already contained in \cite{Ke1}. This is carried out in Sections \ref{subsect:ext} and further developed to deal with morphisms in $\Hqe$ in Section \ref{subsec:morita}. As we will see, this is a conceptually very simple application of the notion of tensor product of dg modules (see \cite{Dr} and \cite{Ke1}). In particular, it should be noted that the core and the really non-trivial part of this paper is the content of Section \ref{sec:extensions}.

This slight change of perspective makes our proof easier also because we can
forget about the model category structure and content ourselves with
the fact that the category of dg categories is a category of fibrant
objects. This has certainly been known for a long time and is
summarized in Section \ref{subsec:morHqe} (after a short introduction
to dg categories in Section \ref{subsec:tensors}).

Once Theorem \ref{thm:Toen} is settled, some important properties of internal Homs which are proved in \cite{To} can be deduced in a straightforward way. This is the case of Corollaries \ref{cor:adj} and \ref{cor:exintHoms2}. The last one covers the dg Morita theory mentioned above.

\smallskip

It is probably worth pointing out here that To\"{e}n's result gave an input
to further generalizations at the level of $\infty$-categories
(see, for example, \cite{BFN}). But for this one really needs the model and simplicial structures on $\dgCat$. This is out of the scope of our paper.

The reader should be also aware that a different approach to the existence of internal Homs was proposed by Tabuada in \cite{Ta}. In particular, he constructs a new model category structure for the homotopy category of dg categories, where the localization takes place with respect to the so-called \emph{Morita equivalences} and not just the quasi-equivalences. Clearly, this means that \cite{Ta} is not in the same generality as \cite{To}. Moreover, \cite{Ta} goes in a transversal direction with respect to the present work. Nevertheless, one can observe that in Tabuada's approach the internal Homs can be naturally interpreted as derived functors.

We conclude this summary going back to the triangulated setting presented at the beginning. It is important to observe that there is no hope that the beauty of Theorem \ref{thm:Toen} can appear in the triangulated context as well. Indeed, it has been shown in \cite{CS2} that the object $\ke\in\Db(X_1\times X_2)$ realizing a Fourier--Mukai functor is by no means unique (up to isomorphism).

\subsection*{Notation and general assumptions}

We denote by $\K$ a commutative ring. By a $\K$-linear category we mean a category
whose Hom spaces are $\K$-modules and such that the compositions
are $\K$-bilinear, not assuming that finite direct sums exist. For a category $\dgA$, we denote by $\Iso(\dgA)$ the set of isomorphism classes of objects in $\dgA$.

Throughout the paper, we assume that a universe $\mathbb{U}$ containing an infinite set is fixed. Anticipating some definitions that will be explained in Section \ref{sec:generalities}, let us spend some words to clarify the context. We will consider \emph{$\mathbb{U}$-small} dg categories, meaning dg categories $\dgD$ such that $\Hom_{\dgD}(D_1,D_2)$ is a complex of $\K$-modules which are isomorphic to objects of $\mathbb{U}$ and such that the collection of objects of $\dgD$ is isomorphic to an object of $\mathbb{U}$ as well. Analogously one can speak about $\mathbb{U}$-small sets. We will then define the dg categories of $\dgD$-dg modules and of h-projective $\dgD$-dg modules. Again, with this we tacitly mean the dg categories of $\mathbb{U}$-small $\dgD$-dg modules and h-projective $\dgD$-dg modules. This
simply refers to the dg modules that take values in the dg category of complexes of $\mathbb{U}$-small $\K$-modules. These are no longer $\mathbb{U}$-small dg categories. Nevertheless, due to the results in \cite[Appendix A]{LO}, they are dg equivalent to $\mathbb{V}$-small dg categories, for some universe
$\mathbb{U}\in\mathbb{V}$. By its definition and main properties (see Section \ref{sec:Morita}), the dg category $\rqr{\farg}$ mentioned above is
\emph{essentially $\mathbb{U}$-small}. This means that the isomorphism classes of its objects form a $\mathbb{U}$-small set. After these warnings and to simplify
the notation, we will not mention explicitly the universe where we are working any longer in the paper, as it should be clear from the context. We refer to \cite[Appendix A]{LO} for all possible subtle logical issues.

\section{Basic properties of dg categories}\label{sec:generalities}

This section collects some very well-known facts concerning dg categories. The emphasis is on the properties of morphisms in the localization of the category of dg categories by quasi-equivalences.

\subsection{Dg categories and tensor product of dg modules}\label{subsec:tensors}

A \emph{dg category} is a $\K$-linear category $\dgA$
such that, for all $A,B\in\Ob(\dgA)$, the morphism spaces
$\dgA(A,B)$ are $\ZZ$-graded $\K$-modules with a differential
$d\colon\dgA(A,B)\to\dgA(A,B)$ of degree $1$ and the composition maps
are morphisms of complexes. A \emph{dg functor} $\dgF\colon\dgA\to\dgB$ between two dg
categories is the datum of a map $\Ob(\dgA)\to\Ob(\dgB)$ and of
morphisms of complexes of $\K$-modules
$\dgA(A,B)\to\dgB(\dgF(A),\dgF(B))$, for all
$A,B\in\Ob(\dgA)$, which are compatible with the compositions and
the units (i.e.\ the identity maps which are, automatically, closed morphisms in degree $0$).
The category with objects dg categories and morphisms dg functors will be denoted by $\dgCat$. Recalled that a dg functor $\dgF\colon\dgA\to\dgB$ is \emph{full} if the
morphisms of complexes of $\K$-modules
$\dgA(A,B)\to\dgB(\dgF(A),\dgF(B))$ are surjective, for all
$A,B\in\Ob(\dgA)$. If such maps are injective, then $\dgF$ is \emph{faithful}.

\begin{ex}\label{ex:catfun}
(i) Every $\K$-linear category can be regarded as a dg category in
which the morphism spaces are concentrated in degree $0$.

(ii) Every dg algebra $A$ over $\K$ defines a dg category with
one object and $A$ as its space of endomorphisms. Notice that an
ordinary $\K$-algebra (in particular, $\K$ itself) can be regarded as
a dg algebra in degree $0$, hence as a dg category with one object.
	
(iii) We denote by $\Cdg$ the dg category whose objects are complexes of $\K$-modules. We refer to \cite[Sect.\ 2.2]{Ke} for the precise definition.

(iv) Given two dg categories $\dgA$ and $\dgB$, one can construct the
dg categories $\dgFun(\dgA,\dgB)$ and $\dgA\otimes\dgB$ (see
\cite[Sect.\ 2.3]{Ke} for the precise definitions). The objects of
$\dgFun(\dgA,\dgB)$ are dg functors from $\dgA$ to $\dgB$ and
morphisms are given by (dg) natural transformations. On the other
hand, the objects of $\dgA\otimes\dgB$ are pairs $(A,B)$ with $A\in\dgA$ and $B\in\dgB$, while the morphisms are defined by
\[
\dgA\otimes\dgB((A_1,B_1),(A_2,B_2))=\dgA(A_1,A_2)\otimes_{\K}\dgB(B_1,B_2),
\]
for all $(A_i,B_i)\in\dgA\otimes\dgB$ and $i=1,2$. Notice that
the tensor product defines a symmetric monoidal structure on $\dgCat$. Namely, up to isomorphism, the tensor product is associative, commutative and $\K$ acts as the identity. It is also easy to see that two dg functors
$\dgF\colon\dgA\to\dgB$ and $\dgG\colon\dgC\to\dgD$ naturally induce a
dg functor $\dgF\otimes\dgG\colon\dgA\otimes\dgC\to\dgB\otimes\dgD$.

(v) If $\dgA$ is a dg category, $\dgA\opp$ denotes the opposite dg
category. The objects of $\dgA\opp$ are the same as those of $\dgA$,
while $\dgA\opp(A,B):=\dgA(B,A)$ and the compositions in $\dgA\opp$
are defined as in $\dgA$, up to a sign (see \cite[Sect.\ 2.2]{Ke} for details).

\end{ex}

If $\dgA$ is a dg category, we denote by $Z^0(\dgA)$ (respectively
$H^0(\dgA)$) the ($\K$-linear) category with the same objects as
$\dgA$ and whose morphisms from $A$ to $B$ are given by $Z^0(\dgA(A,B))$
(respectively $H^0(\dgA(A,B))$). The category $H^0(\dgA)$ is called the
\emph{homotopy category} of $\dgA$, and it has a natural structure of
triangulated category if $\dgA$ is pretriangulated (see \cite[Sect.\ 4.5]{Ke} for
the precise definition). A morphism of $Z^0(\dgA)$ is a {\em dg
isomorphism} (respectively a {\em homotopy equivalence}) if it is an
isomorphism (respectively if its image in $H^0(\dgA)$ is an
isomorphism). Accordingly, two objects $A$ and $B$ of $\dgA$ are {\em
dg isomorphic} (respectively {\em homotopy equivalent}) if $A\iso B$
in $Z^0(\dgA)$ (respectively in $H^0(\dgA)$). If $\dgB$ is another dg
category, two dg functors from $\dgA$ to $\dgB$ will be said to be dg
isomorphic (respectively homotopy equivalent) if they are dg
isomorphic (respectively homotopy equivalent) in $\dgFun(\dgA,\dgB)$.

For all dg categories $\dgA$, $\dgB$ and $\dgC$, there is a natural
isomorphism in $\dgCat$
\begin{equation}\label{tensorHom}
\dgFun(\dgA\otimes\dgB,\dgC)\iso\dgFun(\dgA,\dgFun(\dgB,\dgC)).
\end{equation}
In particular, there is a natural bijection between the objects of the
two dg categories above, which shows that the functor
$\farg\otimes\dgB\colon\dgCat\to\dgCat$ is left adjoint to
$\dgFun(\dgB,\farg)\colon\dgCat\to\dgCat$. Hence the symmetric monoidal structure on $\dgCat$ discussed in Example \ref{ex:catfun} (iv) is closed.

For a dg category $\dgA$, we set $\dgm{\dgA}:=\dgFun(\dgA\opp,\Cdg)$. The objects in $\dgm{\dgA}$ are called \emph{$\dgA$-dg modules}.
Denote by $\hproj{\dgA}$ the full dg subcategory of
$\dgm{\dgA}$ with objects the \emph{h-projective $\dgA$-dg
  modules}. Recall that $M\in\dgm{\dgA}$ is h-projective if
$H^0(\dgm{\dgA})(M,N)=0$, for all $N\in\dgm{\dgA}$ which are acyclic
(meaning that $N(A)$ is an acyclic complex, for all $A\in\dgA$).

\begin{remark}
For every dg category $\dgA$, both $\dgm{\dgA}$ and $\hproj{\dgA}$ are
pretriangulated dg categories, and their homotopy categories are
closed under arbitrary direct sums (see \cite[Sect.\ 2.2]{Ke1}).
\end{remark}

At the same time we denote by $\Pe{\dgA}$ the full dg subcategory of
$\hproj{\dgA}$ consisting of \emph{perfect $\dgA$-dg
  modules}, i.e.\ the compact objects in the triangulated category $H^0(\hproj{\dgA})$.
Recall that an object $C$ in a triangulated category $\cat{D}$ is
\emph{compact} if, given $\{D_i\}_{i\in I}\subset\cat{D}$ such that
$I$ is a set and $\bigoplus_i D_i$ exists in $\cat{D}$, the canonical
map $\bigoplus_i\cat{D}(C,D_i)\tto\cat{D}(C,\oplus_iD_i)$ is an
isomorphism. Moreover, a set of compact objects $\{C_j\}_{j\in
J}\subset\cat{D}$ is a set of \emph{compact generators} for $\cat{D}$
if, given $D\in\cat{D}$ with $\cat{D}(C_j,D[i])=0$ for all $j\in J$
and all $i\in\ZZ$, then $D\iso 0$.

The Yoneda embedding of $\dgA$ is the fully faithful (and injective on
objects) dg functor $\Yon{\dgA}\colon\dgA\to\dgm{\dgA}$ defined on
objects by $\Yon{\dgA}(A):=\dgA(\farg,A)$. The image of $\Yon{\dgA}$
is always contained in $\hproj{\dgA}$. We denote by
$\essim{\dgA}\subseteq\hproj{\dgA}$ the full dg subcategory of
$\dgm{\dgA}$ with objects the dg modules which are homotopy equivalent
to objects in the image of $\Yon{\dgA}$. Notice that $\Yon{\dgA}$ factors through the dg category $\Pe{\dgA}$ which, in turn, contains $\essim{\dgA}$ (see, for example, \cite[Sect.\ 3.5]{Ke}).

Recall that a natural transformation $\theta$ between
two $\dgA$-dg modules $M$ and $N$ is a \emph{quasi-isomorphism} if it
is closed of degree $0$ and $\theta(A)\colon M(A)\to N(A)$ is a
quasi-isomorphism, for every $A\in\dgA$. It can be proved that for
every $M\in\dgm{\dgA}$ there exists an \emph{h-projective resolution}
of $M$, namely a quasi-isomorphism $N\to M$ with $N\in\hproj{\dgA}$
(see \cite[Sect.\ 3.1]{Ke1} and \cite[Sect.\ 14.8]{Dr}). Moreover, a quasi-isomorphism between
two h-projective dg modules is a homotopy equivalence (see, for example,
\cite[Thm.\ 3.4]{KL}).

A dg functor $\dgF\colon\dgA\to\dgB$ induces a functor
$H^0(\dgF)\colon H^0(\dgA)\to H^0(\dgB)$, which is exact
(between triangulated categories) if $\dgA$ and $\dgB$ are
pretriangulated. A dg functor $\dgF\colon\dgA\to\dgB$ is a
\emph{quasi-equivalence}, if the maps
$\dgA(A,B)\to\dgB(\dgF(A),\dgF(B))$ are quasi-isomorphisms, for
every $A,B\in\dgA$, and $H^0(\dgF)$ is an equivalence of categories. One can
consider the localization $\Hqe$ of $\dgCat$ with respect to
quasi-equivalences (which is denoted by $Ho(\dgCat)$ in \cite{To}).

For a dg functor $\dgF$, we write $\hq{\dgF}$ for its image in $\Hqe$ and,
given two dg categories $\dgA$ and $\dgB$, we denote by
$\hqe{\dgA,\dgB}$ the morphisms in $\Hqe$ between $\dgA$ and $\dgB$.
Notice that any $f\in\hqe{\dgA,\dgB}$ induces a ($\K$-linear) functor
$H^0(f)\colon H^0(\dgA)\to H^0(\dgB)$, well defined up to
isomorphism. We say that $H^0(f)$ is \emph{continuous} if it commutes
with arbitrary direct sums in $H^0(\dgA)$. For simplicity, we
sometimes say that $f$ is continuous if $H^0(f)$ is. We denote by
$\hqe{\dgA,\dgB}_c$ the set of continuous morphisms in $\hqe{\dgA,\dgB}$.

\begin{definition}\label{def:termwise}
Let $\dgG_1,\dgG_2\colon\dgA\to\dgB$ be two dg functors.
A natural transformation $\theta\colon\dgG_1\to\dgG_2$ is a \emph{termwise homotopy equivalence} if it is closed of degree $0$ and $\theta(A)\colon\dgG_1(A)\to\dgG_2(A)$ is a homotopy equivalence, for all $A\in\dgA$.	
\end{definition}

\begin{remark}\label{rmk:extnat}
First of all, one observes that the natural transformation in Definition \ref{def:termwise} induces a natural transformation $\theta'$ between $H^0(\dgG_1)$ and $H^0(\dgG_2)$ and $\theta$ is a termwise homotopy equivalence if and only if $\theta'$ is an isomorphism. If $\dgA$ and $\dgB$ are pretriangulated dg categories, the functors $H^0(\dgG_1)$ and $H^0(\dgG_2)$ are exact and the check that $\theta$ is a termwise homotopy equivalence is just a question involving standard techniques in the theory of triangulated categories. 

More precisely, there is indeed a general principle that will be
applied later on. Namely, assume that $\alpha$ is a natural
transformation between two exact continuous functors
$\fun{F}_1,\fun{F}_2\colon\cat{D}\to\cat{D}'$, where $\cat{D}$ and
$\cat{D}'$ are triangulated categories with arbitrary direct sums. Let $\cat{D}_1$ be a full
triangulated subcategory of $\cat{D}$ consisting of compact
generators of $\cat{D}$. Suppose further that, for all
$D\in\cat{D}_1$, we have that $\alpha(D)$ is an isomorphism. Then one
observes that the full subcategory $\cat{D}_2$ of $\cat{D}$
consisting of all objects $D$ such that $\alpha(D)$ is an isomorphism
is obviously triangulated and contains $\cat{D}_1$. On the other hand,
it is easy to see that $\alpha$ is automatically compatible with
arbitrary direct sums, since $\fun{F}_1$ and $\fun{F}_2$ are
continuous. Hence $\cat{D}_2$ is closed under arbitrary direct sums
and so $\cat{D}_2=\cat{D}$ (see \cite{N1}), which proves that
$\alpha$ is an isomorphism.
\end{remark}

\subsection*{Tensor product of dg modules} Following \cite{Dr}, if $\dgA$ is a dg category, the \emph{tensor product} of $M\in\dgm{\dgA}$ and $N\in\dgm{\dgA\opp}$ is defined as
\begin{equation}\label{eqn:tensor}
M\otimes_\dgA N := \cok\left(\Xi\colon
\bigoplus_{A,B \in\dgA} M(B)\otimes_\K\dgA(A,B)\otimes_\K N(A)\tto
\bigoplus_{C \in\dgA} M(C) \otimes_\K N(C)
\right),
\end{equation}
where, given $v_1\in M(B)$ homogeneous of degree $m$, $f\colon A\to B$, homogeneous of degree $n$, and $v_2\in N(A)$, we have
\begin{equation}\label{eqn:tensor1}
\Xi((v_1,f,v_2)):=M(f)(v_1)\otimes v_2-(-1)^{mn}v_1\otimes N(f)(v_2)\in M(A)\otimes_\K N(A)\oplus M(B)\otimes_\K N(B).
\end{equation}
In this version, $M\otimes_\dgA N$ is a complex of $\K$-modules,
hence an object of $\Cdg\iso\dgm{\K}$. Clearly, one can repeat the same definition taking $M\in\dgm{\dgA\otimes\dgB}$ and $N\in\dgm{\dgB\opp\otimes\dgC}$. In this case, $M\otimes_\dgB N$ is an object in $\dgm{\dgA\otimes\dgC}$.

\begin{remark}\label{tensor1}
(i) The definition in \eqref{eqn:tensor} is dg functorial. More precisely, assume we have $M_1,M_2\in\dgm{\dgA\otimes\dgB}$, $N\in\dgm{\dgB\opp\otimes\dgC}$ and a natural transformation $f\colon M_1\to M_2$ of dg functors. Then it is straightforward that this induces a natural transformation $M_1\otimes_\dgB N\to M_2\otimes_\dgB N$ of dg functors.
	
(ii) Given $M\in\dgm{\dgA}$ and $N\in\dgm{\dgB}$, we can think of them
as objects in $\dgm{\dgA\otimes\K}$ and
$N\in\dgm{\K\opp\otimes\dgB}$. Now take the object $M\otimes_\K
N\in\dgm{\dgA\otimes\dgB}$. We claim that, for all
$(A,B)\in\dgA\otimes\dgB$, we get $(M\otimes_\K N)((A,B))=
M(A)\otimes_\K N(B)$, where the tensor product on the
right-hand side is the usual tensor product of complexes of
$\K$-modules. Indeed, as $\K$ consists of only
one object and every morphism is a scalar multiple of the identity,
the map $\Xi$ in \eqref{eqn:tensor1} is trivial in this case, and thus
we get the result from \eqref{eqn:tensor}. In the rest of the paper,
we write $M\otimes N$ for $M\otimes_\K N$.
\end{remark}

\begin{remark}\label{tensor2}
The tensor product of dg modules over a dg category looks very much
like the ordinary tensor product of modules over a (not necessarily
commutative) ring (see, for example, \cite[Sect.\ VI.7]{ML} for the
latter case). In particular, they share many properties, such as the
associativity. Of course they can be proved directly using the
definition. The easy computations are left to the reader, who, on the
other hand, can have a look at \cite[Sect.\ 6]{Mi} for the proof of
associativity in the case of tensor product of modules over a $\ZZ$-linear category.
\end{remark}

For a dg category $\dgA$, an $\dgA$-dg module $M$ is \emph{h-flat} if, for any $N\in\dgm{\dgA\opp}$ which is acyclic, the tensor product
$M\otimes_\dgA N$ is acyclic. One can check that any h-projective dg
module is h-flat and that a dg module which is homotopy equivalent to an h-flat dg module is h-flat itself (see \cite[Sect.\ 3.5]{KL}, for some more details).

\subsection*{Derived tensor product of dg categories}

We say that a dg category $\dgA$ is \emph{h-projective} if $\dgA(A,B)$
is in $\hproj{\K}$, for all $A,B\in\dgA$. It is clear that if $\dgA$
if h-projective, then $\dgA\opp$ is h-projective as well and all dg
categories are h-projective if $\K$ is a field. It is also easy to see
that if $\dgA$ and $\dgB$ are h-projective, then $\dgA\otimes\dgB$ is
h-projective, too. We denote by $\hpdgCat$ the full subcategory of $\dgCat$ consisting of h-projective dg categories.

\begin{remark}\label{rmk:hprojdg1}
Using \cite[Sect.\ 13.5]{Dr} (or the fact that $\dgCat$ is a model category \cite{TaDg}), for any dg category $\dgA$, one can construct
an explicit h-projective dg category $\hpdg{\dgA}$ with a
quasi-equivalence $\dgQ_\dgA\colon\hpdg{\dgA}\to\dgA$, which will be
fixed once and for all. In particular, if $\dgA$ is h-projective we
assume that $\dgQ_{\dgA}=\id_{\dgA}$. If not, following \cite{Dr},
$\hpdg{\dgA}$ can be constructed as a semi-free resolution of $\dgA$,
namely a semi-free dg category $\dgB$ with a quasi-equivalence
$\dgB\to\dgA$. Although it is not needed in the rest of the paper, let us briefly recall that, following \cite[Sect.\ 13.4]{Dr}, a dg category $\dgA$ is \emph{semi-free} if  it can be
represented as the union of an increasing sequence of dg subcategories $\dgA_i$, where $i\in\mathbb{N}$,
such that $\dgA_0$  is a discrete dg category and, for $i>0 $, each $\dgA_i$ is freely generated over $\dgA_{i-1}$, as a graded category, by a family of homogeneous morphisms $f_j$ whose differentials $d(f_j)$
are morphisms in $\dgA_{i-1}$.
It is then a simple calculation to see that $\hpdg{\dgA}(A,B)$ is h-projective for all $A,B\in\hpdg{\dgA}$ (one can also combine \cite[Lemma 13.6]{Dr}, \cite[Prop.\ 2.3 (3)]{To} and the fact that a cofibrant complex of $\K$-modules is h-projective).

Denoting by $\hpHqe$ the localization of $\hpdgCat$ by all
quasi-equivalences, it is then easy to verify that the natural functor
$\hpHqe\to\Hqe$ is an equivalence (one can use \cite[Lemma 13.5]{Dr} for the faithfulness).
\end{remark}

Hence, given two dg categories $\dgA$ and $\dgB$, we define the \emph{derived tensor product} as
\[
\dgA\lotimes\dgB:=\hpdg{\dgA}\otimes\dgB.
\]

\begin{remark}\label{qe}
Given an h-projective dg category $\dgA$ and a quasi-equivalence
$\dgF\colon\dgB\to\dgB'$, the induced dg functor
$\dgG:=\id_{\dgA}\otimes\dgF\colon\dgA\otimes\dgB\to\dgA\otimes\dgB'$ is
again a quasi-equivalence. Indeed, the complex $\dgA(A_1,A_2)$ is h-flat (being h-projective), for all $A_1,A_2\in\dgA$, and thus the maps
\[
\dgA\otimes\dgB((A_1,B_1),(A_2,B_2))\to\dgA\otimes\dgB'(\dgG((A_1,B_1)),\dgG((A_2,B_2)))
\]
are quasi-isomorphisms, for all $A_1,A_2\in\dgA$ and all
$B_1,B_2\in\dgB$, since the maps
$\dgB(B_1,B_2)\to\dgB'(\dgF(B_1),\dgF(B_2))$ are quasi-isomorphisms.
Hence it remains to show that $H^0(\id_{\dgA}\otimes\dgF)$ is essentially surjective. This is clear from the definition of tensor product of dg categories, being $H^0(\dgF)$ essentially surjective.

It follows from the universal property of
the localization of $\dgCat$ by quasi-equivalences that the functor
$\dgA\otimes\farg\colon\dgCat\to\dgCat$ naturally induces a
functor $\dgA\otimes\farg\colon\Hqe\to\Hqe$.

Similarly, given a dg category $\dgB$ and a quasi-equivalence $\dgG\colon\dgA\to\dgA'$, with $\dgA$ and $\dgA'$ h-projective, the induced dg functor $\dgG\otimes\id_{\dgB}\colon\dgA\otimes\dgB\to\dgA'\otimes\dgB$ is again a quasi-equivalence. Hence the functor
$\farg\otimes\dgB\colon\hpdgCat\to\dgCat$ naturally induces a
functor $\farg\otimes\dgB\colon\hpHqe\to\Hqe$.
\end{remark}

Putting Remarks \ref{rmk:hprojdg1} and \ref{qe} together, we get a
well-defined functor $\farg\lotimes\farg\colon\Hqe\times\Hqe\tto\Hqe$
which endows $\Hqe$ with a symmetric monoidal structure.

\subsection{Some properties of morphisms in \texorpdfstring{$\Hqe$}{Hqe}}\label{subsec:morHqe}

The content of this section is probably well known to experts. Moreover, most of the properties of $\dgCat$ and $\Hqe$ mentioned here have trivial proofs when we regard $\dgCat$ as a model category (see \cite{TaDg}). Nonetheless, to achieve a proof of Theorem \ref{thm:Toen} much less is needed and we sketch in this section the minimal amount of information which is required. Trying to keep the paper as much self-contained as possible, we outline the proofs of the results of this section.

Recall from \cite[Sect.\ 1]{Br} the notion of \emph{category of fibrant
objects}. Let $\cat{C}$ be a category with finite products and
assume that $\cat{C}$ has two distinguished classes of morphisms,
called weak equivalences and fibrations. A morphism which is both a
weak equivalence and a fibration will be called a trivial fibration. A
path object for $C\in\cat{C}$ is an object $\gpob{C}$ of $\cat{C}$
together with a weak equivalence $C\to\gpob{C}$ and a fibration
$\gpob{C}\to C\times C$ whose composition is the diagonal $C\to
C\times C$. We say that $\cat{C}$, together with its weak equivalences
and fibrations, is a category of fibrant objects if the following
axioms are satisfied.
\begin{itemize}
\item[(A)] Let $f$ and $g$ be morphisms of $\cat{C}$ such that $g\comp
f$ is defined. If two of the morphisms $f$, $g$ and $g\comp f$ are weak
equivalences, then so is the third. Any isomorphism is a weak
equivalence.
\item[(B)] The composition of two fibrations is a fibration. Any
isomorphism is a fibration.
\item[(C)] Fibrations and trivial fibrations are preserved by base
extension.
\item[(D)] For every $C\in\cat{C}$ there exists at least one path
object $\gpob{C}$.
\item[(E)] For every $C\in\cat{C}$ the morphism from $C$ to a terminal
object is a fibration.
\end{itemize}
Notice that if $\cat{C}$ is a model category, then the full
subcategory of $\cat{C}$ consisting of all fibrant objects is a
category of fibrant objects.

As in \cite[Sect.\ 2]{Br}, we say that two morphisms
$f_0,f_1\colon C\to D$ in a category of fibrant objects are \emph{homotopic}
if there exist a path object $D\to\gpob{D}\mor{(p_0,p_1)}D\times D$
and a morphism $h\colon C\to\gpob{D}$ (called \emph{homotopy}) such that
$f_i=p_i\comp h$, for $i=0,1$.

\begin{remark}\label{stdhtp}
If $g\colon B\to C$ is a weak equivalence and $f_0,f_1\colon C\to D$
are morphisms such that $f_0\comp g$ and $f_1\comp g$ are homotopic,
it follows from \cite[Prop.\ 1]{Br} that for any path object
$\gpob{D}$ for $D$ there exists a weak equivalence $g'\colon B'\to C$
such that $f_0\comp g'$ and $f_1\comp g'$ are homotopic by a homotopy
$B'\to\gpob{D}$.
\end{remark}

Now we want to prove that $\dgCat$ is a category of fibrant objects,
if one takes as weak equivalences the quasi-equivalences and as
fibrations the full dg functors whose $H^0$ is an isofibration (a
functor $\fun{F}\colon\cat{C}\to\cat{D}$ is an isofibration if for
every $C\in\cat{C}$ and every isomorphism $f\colon\fun{F}(C)\isomor D$
in $\cat{D}$ there exists an isomorphism $g\colon C\isomor C'$ in
$\cat{C}$ such that $\fun{F}(g)=f$). To this purpose, we need to fix
some notation.

As in \cite[Sect.\ 2.9]{Dr}, for every dg category $\dgA$ we denote
by $\dgMor{\dgA}$ the dg category whose objects are triples $(A,B,f)$
with $f\in Z^0(\dgA(A,B))$, and whose morphisms are given by
\[
\dgMor{\dgA}((A,B,f),(A',B',f'))^n:=\dgA(A,A')^n\oplus\dgA(B,B')^n
\oplus\dgA(A,B')^{n-1}
\]
for every $n\in\Z$. If $(a,b,h)\in\dgMor{\dgA}((A,B,f),(A',B',f'))^n$,
the differential is defined by
\[
d(a,b,h):=(d(a),d(b),d(h)+(-1)^n(f'\comp a-b\comp f))
\]
and the composition by
$(a',b',h')\comp(a,b,h):=(a'\comp a,b'\comp b,b'\comp h+(-1)^nh'\comp a)$.
Actually we will be interested in the full dg subcategory $\pob{\dgA}$
of $\dgMor{\dgA}$ with objects the triples $(A,B,f)$ such that $f$ is
a homotopy equivalence (see \cite[Sect.\ 3]{Ta}).

Notice that there is a natural dg functor
$\inc{\dgA}\colon\dgA\to\pob{\dgA}$, defined on objects by
$A\mapsto(A,A,\id_A)$ and on morphisms by
$f\mapsto(f,f,0)$. Similarly, there are obvious dg functors
$\src{\dgA},\tar{\dgA}\colon\pob{\dgA}\to\dgA$ (``source'' and
``target'') defined both on objects and morphisms as the projection
respectively on the first and on the second component.

\begin{lem}
With the above defined weak equivalences and fibrations, $\dgCat$ is a
category of fibrant objects.
\end{lem}

\begin{proof}
The result depends on some elementary (but tedious) checks which are left to the reader. We simply outline the main ingredients in the proof.

First of all observe that finite products exist in $\dgCat$, and they are given by
the corresponding products both on objects and on morphisms (with
differentials and compositions defined componentwise). In particular,
a terminal object is the dg category with one object and $0$ as the
space of morphisms.

Axioms (A), (B) and (E) are straightforward to check from the
definitions. As for axiom (C), note that for every dg functors
$\dgF\colon\dgA\to\dgC$ and $\dgG\colon\dgB\to\dgC$ the fibre product
$\dgD:=\dgA\times_{\dgC}\dgB$ along $\dgF$ and $\dgG$ exists in
$\dgCat$, and it is given by the full subcategory of $\dgA\times\dgB$
with objects those $(A,B)\in\dgA\times\dgB$ such that
$\dgF(A)=\dgG(B)$ and morphisms those morphisms $(f,g)$ of
$\dgA\times\dgB$ such that $\dgF(f)=\dgG(g)$. It is easy to show that, if
$\dgG$ is a fibration (resp.\ a trivial fibration), then the projection dg functor
$\dgD\to\dgA$ is a fibration (resp.\ a trivial fibration), too.

Passing to axiom (D), one shows that the dg functors
$\dgA\mor{\inc{\dgA}}\pob{\dgA}\mor{(\src{\dgA},\tar{\dgA})}\dgA\times\dgA$
define a path object for any dg category $\dgA$. More precisely, as the composition is clearly
the diagonal $\dgA\to\dgA\times\dgA$, one just shows that
$\inc{\dgA}$ is a quasi-equivalence and $(\src{\dgA},\tar{\dgA})$ is a
fibration.
\end{proof}

Two dg functors $\dgF,\dgG\colon\dgA\to\dgB$ will be called
\emph{standard homotopic} if there exists a dg functor
$\dgH\colon\dgA\to\pob{\dgB}$ such that $\dgF=\src{\dgB}\comp\dgH$ and
$\dgG=\tar{\dgB}\comp\dgH$. From the results proved in
\cite[Sect.\ 2]{Br} for the localization of an arbitrary category of
fibrant objects with respect to weak equivalences, we
obtain the following properties of morphisms in $\Hqe$.

\begin{prop}\label{Hqemor}
\begin{enumerate}
\item\label{Hqehtp} Given two dg functors $\dgF,\dgG\colon\dgA\to\dgB$, we
have $\hq{\dgF}=\hq{\dgG}$ if and only if there exists a
quasi-equivalence $\dgI\colon\dgA'\to\dgA$ such that $\dgF\comp\dgI$
and $\dgG\comp\dgI$ are (standard) homotopic.
\item\label{Hqecomm} Given two dg functors $\dgF\colon\dgA\to\dgC$ and
$\dgG\colon\dgB\to\dgC$ with $\dgG$ a quasi-equivalence, there exist
dg functors $\dgF'\colon\dgD\to\dgB$ and
$\dgG'\colon\dgD\to\dgA$ with $\dgG'$ a quasi-equivalence such that
$\dgF\comp\dgG'$ and $\dgG\comp\dgF'$ are homotopic (hence
$\hq{\dgF\comp\dgG'}=\hq{\dgG\comp\dgF'}$ by part
\eqref{Hqehtp}). If moreover $\dgF$ is a quasi-equivalence, then
$\dgF'$ is a quasi-equivalence, too.
\item\label{Hqefrac} Given morphisms $f_i\colon\dgA\to\dgB$ in
$\Hqe$, for $i=1,\dots,n$, there exist dg functors $\dgI\colon\dgA'\to\dgA$ and
$\dgF_i\colon\dgA'\to\dgB$ with $\dgI$ a quasi-equivalence such that
$f_i=\hq{\dgF_i}\comp\hq{\dgI}^{-1}$, for $i=1,\dots,n$.
\end{enumerate}
\end{prop}

\begin{proof}
Taking into account Remark \ref{stdhtp}, (1) follows from part
(ii) of \cite[Thm.\ 1]{Br}.
The first part of (2) follows from \cite[Prop.\ 2]{Br}. The last
statement in (2) is then an easy consequence of axiom (A), using the fact
that a morphism homotopic to a weak equivalence is also a weak
equivalence (to see this, one uses again axiom (A)).
For $n=1$, (3) follows from part (i) of \cite[Thm.\ 1]{Br}. In
general, one can easily reduce by induction to $n=2$. Then, by the
case $n=1$, for $i=1,2$ there exist dg functors $\dgI_i\colon\dgA_i\to\dgA$ and
$\dgG_i\colon\dgA_i\to\dgB$, with $\dgI_i$ a quasi-equivalence such that
$f_i=\hq{\dgG_i}\comp\hq{\dgI_i}^{-1}$. On the other hand, by part \eqref{Hqecomm}, there exist two quasi-equivalences $\dgJ_i\colon\dgA'\to\dgA_i$ such that $\hq{\dgI_1}\comp\hq{\dgJ_1}=\hq{\dgI_2}\comp\hq{\dgJ_2}$. Set $\dgF_i:=\dgG_i\comp\dgJ_i$ and $\dgI:=\dgI_1\comp\dgJ_1$. It is then clear that $f_i=\hq{\dgF_i}\comp\hq{\dgI}^{-1}$, for $i=1,2$.
\end{proof}

\begin{cor}\label{rephtp}
Let $\dgF,\dgG\colon\dgA\to\dgB$ be dg functors and assume that there
exists a termwise homotopy equivalence $\alpha\colon\dgF\to\dgG$ (in
particular, this is the case if $\dgF$ and $\dgG$ are homotopy
equivalent). Then $\hq{\dgF}=\hq{\dgG}\in\Hqe(\dgA,\dgB)$.
\end{cor}

\begin{proof}
The assumption on $\alpha$ implies that there is a dg functor
$\dgH\colon\dgA\to\pob{\dgB}$ defined on objects by
$A\mapsto(\dgF(A),\dgG(A),\alpha(A))$ and on morphisms by
$a\mapsto(\dgF(a),\dgG(a),0)$. As $\dgH$ clearly gives a standard
homotopy between $\dgF$ and $\dgG$, we conclude that
$\hq{\dgF}=\hq{\dgG}$ by part \eqref{Hqehtp} of Proposition \ref{Hqemor}.
\end{proof}

\section{Extensions of morphisms in $\Hqe$ and bimodules}\label{sec:extensions}

In this section we develop the key ingredients in our proof of Theorem \ref{thm:Toen}. As it turns out, they rely on some natural properties of extension of dg functors. Finally, we provide an interpretation of the morphisms in $\Hqe$ in terms of dg modules over tensor dg categories.

\subsection{Extensions of dg functors}\label{subsect:ext}
Given two dg categories $\dgA$ and $\dgB$, by \eqref{tensorHom} there
is an isomorphism of dg categories
$\dgm{\dgA\opp\otimes\dgB}\iso\dgFun(\dgA,\dgm{\dgB})$, so in
particular an object $\kke\in\dgm{\dgA\opp\otimes\dgB}$ corresponds to
a dg functor $\dfun{\kke}\colon\dgA\to\dgm{\dgB}$. Conversely, for
every dg functor $\dgF\colon\dgA\to\dgm{\dgB}$ there exists a unique
$\kke\in\dgm{\dgA\opp\otimes\dgB}$ such that $\dfun{\kke}=\dgF$. An
object $\kke\in\hproj{\dgA\opp\otimes\dgB}$ is called \emph{right
quasi-representable} if $\dfun{\kke}(\dgA)\subset\essim{\dgB}$. The
full dg subcategory of $\hproj{\dgA\opp\otimes\dgB}$ consisting of all
right quasi-representable dg modules will be denoted by
$\rqr{\dgA\opp\otimes\dgB}$. Notice that $\rqr{\dgA\opp\otimes\dgB}$
is always closed under homotopy equivalences and, if $\dgA$ is the dg category
$\K$, then it is isomorphic to $\essim{\dgB}$.

Let $\dgF\colon\dgA\to\dgm{\dgB}$ be a dg functor corresponding to
$\kke\in\dgm{\dgA\opp\otimes\dgB}$. Following \cite[Sect.\ 6.1]{Ke1}, we define the \emph{extension} of $\dgF$ to be the dg functor
\[
\ext{\dgF}\colon\dgm{\dgA}\tto\dgm{\dgB}\qquad \ext{\dgF}(\farg):=\farg\otimes_{\dgA}\kke.
\]
Notice that the definition of $\ext{\dgF}$ is a reformulation of the usual notion of Kan extension in the context of dg functors. There is also a natural dg functor
\[
\res{\dgF}\colon\dgm{\dgB}\to\dgm{\dgA}\qquad\res{\dgF}(M):=\dgm{\dgB}(\dgF(\farg),M),
\]
for every $M\in\dgm{\dgB}$.

\begin{remark}
Notice that, in \cite{Ke1}, the dg functors $\ext{\dgF}$ and $\res{\dgF}$ above are denoted by $\mathrm{T}_\kke$ and $\mathrm{H}_\kke$ respectively.
\end{remark}

If $\dgG\colon\dgA\to\dgB$
is a dg functor, $\ext{\Yon{\dgB}\comp\dgG}$ is usually denoted by
$\Ind{\dgG}\colon\dgm{\dgA}\to\dgm{\dgB}$, whereas (due to the dg
version of Yoneda's
lemma) $\res{\Yon{\dgB}\comp\dgG}$ is dg isomorphic to the dg functor
$\Res{\dgG}\colon\dgm{\dgB}\to\dgm{\dgA}$ defined by
$\Res{\dgG}(M):=M(\dgG(\farg))$.

\begin{prop}\label{extres}
Let $\dgF\colon\dgA\to\dgm{\dgB}$ and $\dgG\colon\dgA\to\dgB$ be dg
functors.
\begin{enumerate}
\item\label{extresa} $\ext{\dgF}$ is left adjoint to $\res{\dgF}$
(hence $\Ind{\dgG}$ is left adjoint to $\Res{\dgG}$).
\item\label{extresi} $\ext{\dgF}\comp\Yon{\dgA}$ is dg isomorphic to
$\dgF$ and $H^0(\ext{\dgF})$ is continuous (hence
$\Ind{\dgG}\comp\Yon{\dgA}$ is dg isomorphic to
$\Yon{\dgB}\comp\dgG$ and $H^0(\Ind{\dgG})$ is continuous).
\item\label{extresh} $\ext{\dgF}(\hproj{\dgA})\subseteq\hproj{\dgB}$
if and only if $\dgF(\dgA)\subseteq\hproj{\dgB}$ (hence
$\Ind{\dgG}(\hproj{\dgA})\subseteq\hproj{\dgB}$).
\item\label{extresr} $\Res{\dgG}(\hproj{\dgB})\subseteq\hproj{\dgA}$
if and only if $\Res{\dgG}(\essim{\dgB})\subseteq\hproj{\dgA}$;
moreover, $H^0(\Res{\dgG})$ is always continuous.
\item\label{extresq} $\Ind{\dgG}\colon\hproj{\dgA}\to\hproj{\dgB}$ is
a quasi-equivalence if $\dgG$ is a quasi-equivalence.
\end{enumerate}
\end{prop}

\begin{proof}
All the statements above are probably well known (see \cite{Ke1}). Thus we simply
sketch the main ingredients in the proofs. The proof of
\eqref{extresa} uses exactly the same argument as in
\cite[Sect.\ 14.9]{Dr} for the adjunction between $\Ind{\dgG}$ and $\Res{\dgG}$.
The first part of \eqref{extresi} follows from Eq.\ (14.2) in
\cite{Dr}, and $H^0(\ext{\dgF})$ is continuous because it is left
adjoint to $H^0(\res{\dgF})$ by \eqref{extresa}.

The non-trivial
implication in \eqref{extresh} is a consequence of \eqref{extresi} and
of $\dgF(\dgA)\subseteq\hproj{\dgB}$. Indeed, we use here that the
objects of $\dgA$ form a set of compact generators for $\hproj{\dgA}$
(see \cite[Sect.\ 4.2]{Ke1}) and that $H^0(\dgF)$ is continuous by
\eqref{extresi}, arguing exactly as at the end of Remark \ref{rmk:extnat}.

A similar argument applies to
the first part of \eqref{extresr}. For the fact that $H^0(\Res{\dgG})$
is continuous, we use that $\Res{\dgG}$ has a right adjoint (see, for
example, \cite[Sect.\ 1]{LO}). Finally, \eqref{extresq} is observed in \cite[Remark 4.3]{Dr}.
\end{proof}

Let $\dgh{\dgA\opp\otimes\dgB}$ be the full dg subcategory of
$\dgm{\dgA\opp\otimes\dgB}$ with objects $\kke$ such that
$\dfun{\kke}(\dgA)\subseteq\hproj{\dgB}$. Denoting by
$\dgid{\dgA}\in\dgm{\dgA\opp\otimes\dgA}$ the object such that
$\dfun{\dgid{\dgA}}=\Yon{A}\colon\dgA\to\dgm{\dgA}$, obviously
$\dgid{\dgA}\in\dgh{\dgA\opp\otimes\dgA}$.

\begin{remark}\label{qiso}
If $\alpha\colon\kke\to\kke'$ is a quasi-isomorphism in
$\dgh{\dgA\opp\otimes\dgB}$, then clearly
$\dfun{\alpha}\colon\dfun{\kke}\to\dfun{\kke'}$ is such that
$\dfun{\alpha}(A)$ is a quasi-isomorphism in $\hproj{\dgB}$, and hence
a homotopy equivalence, for every $A\in\dgA$. In other words, $\dfun{\alpha}$ is a termwise
homotopy equivalence.
\end{remark}

\begin{lem}\label{hproj}
The inclusion $\hproj{\dgA\opp\lotimes\dgB}\subseteq\dgh{\dgA\opp\lotimes\dgB}$ holds.
\end{lem}

\begin{proof}
First, we can assume, without loss of generality, that $\dgA$ is h-projective and work with $\dgA\opp\otimes\dgB$. Then, for an object $A$ of $\dgA$, consider the inclusion
$\dgI_A\colon\dgB\to\dgA\opp\otimes\dgB$ defined as $\dgI_A(B):=(A,B)$, for
all $B$ in $\dgB$. Now, given $\kke\in\dgm{\dgA\opp\otimes\dgB}$ and
$A$ in $\dgA$, we have $\dfun{\kke}(A)=\Res{\dgI_A}(\kke)$. Thus, it is
enough to observe that
$\Res{\dgI_A}(\hproj{\dgA\opp\otimes\dgB})\subseteq\hproj{\dgB}$. For
this, one applies part \eqref{extresr} of Proposition \ref{extres},
since $\Res{\dgI_A}(\Yon{\dgA\opp\otimes\dgB}(A',B))\iso\dgA(A',A)\otimes
\Yon{\dgB}(B)\in\hproj{\dgB}$ (thanks to the fact that
$\dgA(A',A)\in\hproj{\K}$, $\dgA$ being h-projective), for all $(A',B)$ in
$\dgA\opp\otimes\dgB$.
\end{proof}

\begin{lem}\label{extfun}
The map $\kke\mapsto\dgFM{\kke}$ extends to a dg functor
\[
\dgm{\dgA\opp\otimes\dgB}\tto\dgFun(\dgm{\dgA},\dgm{\dgB}),
\]
which
restricts to a dg functor
$\dgh{\dgA\opp\otimes\dgB}\to\dgFun(\hproj{\dgA},\hproj{\dgB})$.
\end{lem}

\begin{proof}
The first part of the statement is a simple consequence of the fact that, by definition, the tensor product of dg modules is functorial (see Remark \ref{tensor1} (i)). For the second part, observe that, by definition, $\dfun{\kke}(\dgA)\subseteq\hproj{\dgB}$, when $\kke\in\dgh{\dgA\opp\otimes\dgB}$. Hence we apply part \eqref{extresh} of Proposition \ref{extres}.
\end{proof}

Lemma \ref{extfun} implies that, given a dg functor $\dgF\colon\dgA\to\hproj{\dgB}$, we can think of the extension of $\dgF$
 as a dg functor $\ext{\dgF}\colon\hproj{\dgA}\to\hproj{\dgB}$.

\begin{lem}\label{extcomp}
Given two dg functors $\dgF\colon\dgA\to\dgm{\dgB}$ and
$\dgG\colon\dgB\to\dgm{\dgC}$, there is a dg isomorphism of dg
functors
\[\ext{\ext{\dgG}\comp\dgF}\iso\ext{\dgG}\comp\ext{\dgF}\colon
\dgm{\dgA}\to\dgm{\dgC}.\]
Moreover, if $\dgF'\colon\dgA\to\dgB$ and $\dgG'\colon\dgB\to\dgC$ are
two other dg functors, then there are also dg isomorphisms
$\ext{\dgG\comp\dgF'}\iso\ext{\dgG}\comp\Ind{\dgF'}$ and
$\ext{\Ind{\dgG'}\comp\dgF}\iso\Ind{\dgG'}\comp\ext{\dgF}$.
\end{lem}

\begin{proof}
Let $\kkf\in\dgm{\dgA\opp\otimes\dgB}$ and
$\kkg\in\dgm{\dgB\opp\otimes\dgC}$ be such that $\dgF=\dfun{\kkf}$ and
$\dgG=\dfun{\kkg}$. Then $\ext{\dgG}\comp\ext{\dgF}\iso\ext{\dgH}$,
where $\dgH:=\dfun{\kkf\otimes_{\dgB}\kkg}$, by the associativity of
the tensor product. Using part \eqref{extresi} of Proposition
\ref{extres}, it follows that
\[\dgH\iso\ext{\dgH}\comp\Yon{\dgA}\iso
\ext{\dgG}\comp\ext{\dgF}\comp\Yon{\dgA}\iso\ext{\dgG}\comp\dgF,\]
which proves the first part. The last statement then follows taking
$\dgF=\Yon{\dgB}\comp\dgF'$ or $\dgG=\Yon{\dgC}\comp\dgG'$ and using again part \eqref{extresi} of Proposition
\ref{extres}.
\end{proof}

\begin{lem}\label{kerprod}
Given dg categories $\dgA_i$, $\dgB_i$ and objects
$\kke_i\in\dgm{\dgA_i\opp\otimes\dgB_i}$ for $i=1,2$, the diagram
\[\xymatrix{\dgA_1\otimes\dgA_2
\ar[rrr]^-{\dfun{\kke_1}\otimes\dfun{\kke_2}}
\ar[rrrd]_-{\dfun{\kke_1\otimes\kke_2}} & & &
\dgm{\dgB_1}\otimes\dgm{\dgB_2} \ar[d]^{\farg\otimes\farg} \\
 & & & \dgm{\dgB_1\otimes\dgB_2}}\]
(where $\kke_1\otimes\kke_2\in
\dgm{\dgA_1\opp\otimes\dgB_1\otimes\dgA_2\opp\otimes\dgB_2}\iso
\dgm{(\dgA_1\otimes\dgA_2)\opp\otimes(\dgB_1\otimes\dgB_2})$) commutes
in $\dgCat$ up to dg isomorphism. Moreover,
$\dgFM{\kke_1\otimes\kke_2}(\farg)\iso
\kke_1\otimes_{\dgA_1\opp}\farg\otimes_{\dgA_2}\kke_2$.
\end{lem}

\begin{proof}
The commutativity of the diagram follows directly from the fact that $\dfun{\kke_i}(A)=\kke_i((A,\farg))$ and Remark \ref{tensor1} (ii). The second part amounts to showing that, for all $M\in\dgm{\dgA_1\otimes\dgA_2}$, we have the isomorphism $M\otimes_{\dgA_1\otimes\dgA_2}(\kke_1\otimes\kke_2)\iso \kke_1\otimes_{\dgA_1\opp}M\otimes_{\dgA_2}\kke_2$. This is an easy exercise using the definition \eqref{eqn:tensor}.
\end{proof}

\begin{prop}\label{prop:rqr}
Let $\dgF\colon\dgA\to\dgA'$ and $\dgG\colon\dgB\to\dgB'$ be dg
functors with $\dgA$ and $\dgA'$ h-projective.
\begin{enumerate}
\item The dg functor $\dgF$ induces a natural map
$\Iso(H^0(\rqr{{\dgA'}\opp\otimes\dgB}))\to
\Iso(H^0(\rqr{\dgA\opp\otimes\dgB}))$;
if moreover $\dgF$ is a quasi-equivalence, then this map is bijective
and $\Ind{\dgF\opp\otimes\id_{\dgB}}$ restricts to a quasi-equivalence
$\rqr{\dgA\opp\otimes\dgB}\to\rqr{{\dgA'}\opp\otimes\dgB}$.
\item The dg functor $\Ind{\id_{\dgA\opp}\otimes\dgG}$ restricts to a dg
functor $\rqr{\dgA\opp\otimes\dgB}\to\rqr{\dgA\opp\otimes\dgB'}$,
which is a quasi-equivalence if $\dgG$ is such.
\end{enumerate}
\end{prop}

\begin{proof}
As for (1), notice that, setting
$\dgF_1:=\dgF\opp\otimes\id_{\dgB}$, the dg functor
$\Ind{\dgF_1}\colon\hproj{\dgA\opp\otimes\dgB}\to
\hproj{{\dgA'}\opp\otimes\dgB}$ clearly induces a natural map
$\Iso(H^0(\hproj{\dgA\opp\otimes\dgB}))\to
\Iso(H^0(\hproj{{\dgA'}\opp\otimes\dgB}))$.
On the other hand, one can also define a natural map
$\Iso(H^0(\hproj{{\dgA'}\opp\otimes\dgB}))\to
\Iso(H^0(\hproj{\dgA\opp\otimes\dgB}))$
by $\ho{\kke'}\mapsto\ho{\kke}$, where $\kke$ is an h-projective
resolution of $\Res{\dgF_1}(\kke')$. It is not difficult to show that, if
$\dgF$ (hence $\Ind{\dgF_1}$, by Remark \ref{qe} and part
\eqref{extresq} of Proposition \ref{extres}) is a quasi-equivalence,
then these two maps are bijective and inverse to each other (see, for
example, \cite[Sect.\ 14.12]{Dr}). Therefore it is enough to prove that
$\kke'\in\rqr{{\dgA'}\opp\otimes\dgB}$ implies
$\kke\in\rqr{\dgA\opp\otimes\dgB}$, and that the converse is true if
$\dgF$ is a quasi-equivalence. To see this, observe that the
quasi-isomorphism $\kke\to\Res{\dgF_1}(\kke')$ induces, for every
$A\in\dgA$, a quasi-isomorphism
$\dfun{\kke}(A)\to\dfun{\Res{\dgF_1}(\kke')}(A)=\dfun{\kke'}(\dgF(A))$,
which is in fact a homotopy equivalence (due to the fact that both the
source and the target are in $\hproj{\dgB}$ by Lemma \ref{hproj}). It
follows that $\dfun{\kke}(A)\in\essim{\dgB}$ if and only if
$\dfun{\kke'}(\dgF(A))\in\essim{\dgB}$, which is enough to conclude.
Indeed, by the definition of $\dfun{\kke}$, we have that $\kke\in\rqr{{\dgA}\opp\otimes\dgB}$ if and only if $\dfun{\kke}(A)\in\essim{\dgB}$, for all $A\in\dgA$. On the other hand, it is clear that $\dfun{\kke}(A)\in\essim{\dgB}$ because $\kke'\in\rqr{{\dgA'}\opp\otimes\dgB}$ and thus $\dfun{\kke'}(\dgF(A))\in\essim{\dgB}$, for all $A\in\dgA$.
Clearly, if $\dgF$ is a quasi-equivalence, the same argument shows that  $\kke'\in\rqr{{\dgA'}\opp\otimes\dgB}$ if  $\kke\in\rqr{{\dgA}\opp\otimes\dgB}$.

As for (2), we just need to show that, setting
$\dgG_1:=\id_{\dgA\opp}\otimes\dgG$, the dg functor
$\Ind{\dgG_1}\colon\hproj{\dgA\opp\otimes\dgB}\to
\hproj{\dgA\opp\otimes\dgB'}$ sends $\rqr{\dgA\opp\otimes\dgB}$ to
$\rqr{\dgA\opp\otimes\dgB'}$ (because
then the second part of the statement can be proved with an argument
which is completely similar to the one used in (1)). Given
$\kke\in\hproj{\dgA\opp\otimes\dgB}$ and setting
$\kke':=\Ind{\dgG_1}(\kke)\in\hproj{\dgA\opp\otimes\dgB'}$, we claim
that $\dfun{\kke'}\iso\Ind{\dgG}\comp\dfun{\kke}$. Notice that this is
enough to conclude that $\kke'\in\rqr{\dgA\opp\otimes\dgB'}$ if
$\kke\in\rqr{\dgA\opp\otimes\dgB}$, as clearly
$\Ind{\dgG}(\essim{\dgB})\subseteq\essim{\dgB'}$. Now, denoting by
$\kkg_1\in\dgm{(\dgA\opp\otimes\dgB)\opp\otimes(\dgA\opp\otimes\dgB')}
\iso\dgm{\dgA\otimes\dgA\opp\otimes\dgB\opp\otimes\dgB'}$ the dg
module such that
$\dfun{\kkg_1}=\Yon{\dgA\opp\otimes\dgB'}\comp\dgG_1$, it is easy to
see that $\kkg_1\iso\dgid{\dgA\opp}\otimes\kkg$, where
$\kkg\in\dgm{\dgB\opp\otimes\dgB'}$ denotes the dg module such that
$\dfun{\kkg}=\Yon{\dgB'}\comp\dgG$. As $\Ind{\dgG_1}=\dgFM{\kkg_1}$,
by Lemma \ref{kerprod} we get
\[
\kke'=\dgFM{\kkg_1}(\kke)\iso\dgFM{\dgid{\dgA\opp}\otimes\kkg}(\kke)
\iso\dgid{\dgA\opp}\otimes_{\dgA}\kke\otimes_{\dgB}\kkg\iso
\kke\otimes_{\dgB}\kkg,
\]
where the last isomorphism is due to \cite[Sect.\ 14.6]{Dr}. It follows from the associativity of the tensor product that
$\dgFM{\kke'}\iso\dgFM{\kkg}\comp\dgFM{\kke}=
\Ind{\dgG}\comp\dgFM{\kke}$, which implies that
$\dfun{\kke'}\iso\Ind{\dgG}\comp\dfun{\kke}$ by part \eqref{extresi}
of Proposition \ref{extres}.
\end{proof}

\subsection{Extending morphisms in $\Hqe$}\label{subsec:morita}

Let $\dgA$ and $\dgB$ be dg categories.

\begin{lem}\label{lem:extension}
If $\dgF_1,\dgF_2\colon\dgA\to\hproj{\dgB}$ are dg functors such that
$\hq{\dgF_1}=\hq{\dgF_2}$, then $\hq{\ext{\dgF_1}}=\hq{\ext{\dgF_2}}$
in $\hqe{\hproj{\dgA},\hproj{\dgB}}$.
\end{lem}

\begin{proof}
As $\hq{\dgF_1}=\hq{\dgF_2}$, by part
\eqref{Hqehtp} of Proposition \ref{Hqemor}
there exists a quasi-equivalence $\dgI:\dgC\to\dgA$ such that
$\dgG_i:=\dgF_i\comp\dgI$ for $i=1,2$
sit in the commutative diagram
\begin{equation*}\label{eqn:eqaldiagr}
\xymatrix{
&&\dgC\ar[d]^-{\dgH}\ar[dll]_-{\dgG_1}\ar[drr]^-{\dgG_2}&&\\
\hproj{\dgB}&&\pob{\hproj{\dgB}}\ar[ll]^-{\src{\hproj{\dgB}}}\ar[rr]_-{\tar{\hproj{\dgB}}} &&
\hproj{\dgB}
}
\end{equation*}
for some dg functor $\dgH\colon\dgC\to\pob{\hproj{\dgB}}$.
Thus, by Lemma \ref{extcomp}, we have
\begin{equation}\label{eqn:eqi1}
\ext{\dgG_1}=\ext{\src{\hproj{\dgB}}\comp\dgH}\iso\ext{\src{\hproj{\dgB}}}\comp\Ind{\dgH}\qquad\ext{\dgG_2}=\ext{\tar{\hproj{\dgB}}\comp\dgH}\iso\ext{\tar{\hproj{\dgB}}}\comp\Ind{\dgH}.
\end{equation}
Now observe that, by definition, $\src{\hproj{\dgA}}\comp\inc{\hproj{\dgA}}=\id_{\hproj{\dgA}}=\tar{\hproj{\dgA}}\comp\inc{\hproj{\dgA}}$. Hence
\[
\ext{\src{\hproj{\dgA}}\comp\inc{\hproj{\dgA}}}\iso\ext{\src{\hproj{\dgA}}}\comp\Ind{\inc{\hproj{\dgA}}}\iso\ext{\tar{\hproj{\dgA}}\comp\inc{\hproj{\dgA}}}\iso\ext{\tar{\hproj{\dgA}}}\comp\Ind{\inc{\hproj{\dgA}}}
\]
where the first and the last isomorphisms are again due to Lemma
\ref{extcomp}. Then
$\hq{\ext{\src{\hproj{\dgA}}}}\comp\hq{\Ind{\inc{\hproj{\dgA}}}}=\hq{\ext{\tar{\hproj{\dgA}}}}\comp\hq{\Ind{\inc{\hproj{\dgA}}}}$. Since
$\inc{\hproj{\dgA}}$ (and thus, by part \eqref{extresq} of Proposition \ref{extres}, $\Ind{\inc{\hproj{\dgA}}}$) is a quasi-equivalence, we get $\hq{\ext{\src{\hproj{\dgA}}}}=\hq{\ext{\tar{\hproj{\dgA}}}}$.

Using this and \eqref{eqn:eqi1}, we obtain
\[
\hq{\ext{\dgG_1}}=\hq{\ext{\src{\hproj{\dgB}}}}\comp\hq{\Ind{\dgH}}=\hq{\ext{\tar{\hproj{\dgB}}}}\comp\hq{\Ind{\dgH}}=\hq{\ext{\dgG_2}}.
\]
Again by Lemma \ref{extcomp}, we have $\hq{\ext{\dgG_i}}=\hq{\ext{\dgF_i\comp\dgI}}=\hq{\ext{\dgF_i}\comp\Ind{\dgI}}=\hq{\ext{\dgF_i}}\comp\hq{\Ind{\dgI}}$,
for $i=1,2$. As $\Ind{\dgI}$ is a quasi-equivalence by part \eqref{extresq} of Proposition \ref{extres}, the identity
$\hq{\ext{\dgG_1}}=\hq{\ext{\dgG_2}}$ implies $\hq{\ext{\dgF_1}}=\hq{\ext{\dgF_2}}$.
\end{proof}

\begin{prop}\label{prop:Morita}
If $\dgA$ and $\dgB$ are dg categories, the natural map of sets
\begin{equation*}\label{eqn:Morita}
\hqe{\hproj{\dgA},\hproj{\dgB}}_c\tto\hqe{\dgA,\hproj{\dgB}}\qquad
f\mapsto f\comp\hq{\Yon{\dgA}}
\end{equation*}
(where $\Yon{\dgA}$ denotes the Yoneda embedding
$\dgA\to\hproj{\dgA}$) is surjective.\footnote{The map is also inective, as we will see in Corollary \ref{cor:bij}. The proof of injectivity in the published version of the paper is not correct because it uses the same wrong argument used in the proof of \cite[Proposition 1.17]{LO} to show that $H^0(\Phi)$ commutes with direct sums. It is important to observe that, as the same argument actually proves that $H^0(\phi)$ commutes with direct sums, \cite[Proposition 1.17]{LO} is still true.}
\end{prop}

\begin{proof}
Given $f\in\hqe{\dgA,\hproj{\dgB}}$, by part \eqref{Hqefrac} of
Proposition \ref{Hqemor} there exist a quasi-equivalence
$\dgI\colon\dgC\to\dgA$ and a dg functor
$\dgF\colon\dgC\to\hproj{\dgB}$ such that
$f=\hq{\dgF}\comp\hq{\dgI}^{-1}$. As $\Ind{\dgI}$ is a
quasi-equivalence by part \eqref{extresq} of Proposition \ref{extres},
we can define $\ext{f}:=\hq{\ext{\dgF}}\comp\hq{\Ind{\dgI}}^{-1}\in
\hqe{\hproj{\dgA},\hproj{\dgB}}$. By
part \eqref{extresi} of Proposition \ref{extres} we see that
$\ext{f}\in\hqe{\hproj{\dgA},\hproj{\dgB}}_c$ and
$\ext{f}\comp\hq{\Yon{\dgA}}=f$, thereby proving that the map is
surjective.
\end{proof}

On the other hand, given three dg categories $\dgA$, $\dgB$ and $\dgC$ with a fully
faithful dg functor $\dgJ\colon\dgB\to\dgC$, we have another natural map of sets
\begin{equation*}\label{eqn:Morita2}
\monomap[\dgA,\dgJ]\colon\hqe{\dgA,\dgB}\to\hqe{\dgA,\dgC}\qquad f\mapsto\hq{\dgJ}\comp f.
\end{equation*}

\begin{prop}\label{prop:Morita2}
The natural map of sets $\monomap[\dgA,\dgJ]$ is injective.
\end{prop}

\begin{proof}
Set $\dgC'$ to be the full dg subcategory of $\dgC$ consisting
of all objects in the essential image of $H^0(\dgJ)$ and denote by $\dgJ_1\colon\dgB\to\dgC'$
the natural quasi-equivalence. Let $\dgJ_2\colon\dgC'\to\dgC$ be the
natural inclusion inducing a natural dg functor
$\dgJ_3\colon\pob{\dgC'}\to\pob{\dgC}$ such that
$\dgJ_3((C_1,C_2,f))=(\dgJ_2(C_1),\dgJ_2(C_2),\dgJ_2(f))$ for every
$(C_1,C_2,f)\in\pob{\dgC'}$. It is easy to verify that $\dgJ_3$ is
fully faithful since $\dgJ_2$ is.

Given $f_1,f_2\in\hqe{\dgA,\dgB}$ such that
$\monomap[\dgA,\dgJ](f_1)=\monomap[\dgA,\dgJ](f_2)$, by part
\eqref{Hqefrac} of Proposition \ref{Hqemor} there exist a
quasi-equivalence $\dgI\colon\dgD\to\dgA$ and dg functors
$\dgF_i\colon\dgD\to\dgB$ such that
$f_i=\hq{\dgF_i}\comp\hq{\dgI}^{-1}$, for $i=1,2$. As
$\hq{\dgJ\comp\dgF_1}=\hq{\dgJ\comp\dgF_2}$, by part \eqref{Hqehtp} of
Proposition \ref{Hqemor} there exist
a quasi-equivalence $\dgI':\dgD'\to\dgD$ and a dg functor
$\dgH\colon\dgD'\to\pob{\dgC}$ such that, setting
$\dgG_i:=\dgJ\comp\dgF_i\comp\dgI'\colon\dgD'\to\dgC$ for $i=1,2$,
$\dgG_1=\src{\dgC}\comp\dgH$ and $\dgG_2=\tar{\dgC}\comp\dgH$. Observe
that, by definition, $\dgH(D)=(\dgG_1(D),\dgG_2(D),f)$, where
$f\colon\dgG_1(D)\to\dgG_2(D)$ is a homotopy equivalence, for all $D\in\dgD$.
It is easy to see that $\dgH$ factors through $\dgJ_3$. This means that there exists a dg functor $\dgH'\colon\dgD'\to\pob{\dgC'}$ such that $\dgH=\dgJ_3\comp\dgH'$. Thus, if we set $\dgG'_i=\dgJ_1\comp\dgF_i\comp\dgI'$, we get $\dgG'_1=\src{\dgC'}\comp\dgH'$ and $\dgG'_2=\tar{\dgC'}\comp\dgH'$, so that $\hq{\dgG'_1}=\hq{\dgG'_2}$.

Therefore $\hq{\dgJ_1}\comp\hq{\dgF_1}\comp\hq{\dgI'}=\hq{\dgG'_1}=
\hq{\dgG'_2}=\hq{\dgJ_1}\comp\hq{\dgF_2}\comp\hq{\dgI'}$
and, using that $\dgJ_1$ and $\dgI'$ are quasi-equivalences, we
conclude that $\hq{\dgF_1}=\hq{\dgF_2}$, whence $f_1=f_2$.
\end{proof}

\subsection{Morphisms in $\Hqe$ as dg modules}\label{subsec:genFM}

If $\kke,\kke'\in\dgh{\dgA\opp\lotimes\dgB}$ are quasi-isomorphic, then
it follows from Remark \ref{qiso} and Corollary \ref{rephtp} that
$\hq{\dfun{\kke}}=\hq{\dfun{\kke'}}\colon\hpdg{\dgA}\to\hproj{\dgB}$. In
particular, denoting by
$\ho{\kke}\in\Iso(H^0(\hproj{\dgA\opp\lotimes\dgB}))$ the homotopy
equivalence class of $\kke\in\hproj{\dgA\opp\lotimes\dgB}$ and composing with the natural bijection between $[\dgA,\dgB]$ and $[\hpdg{\dgA},\dgB]$ induced by the quasi-equivalence $\dgQ_\dgA\colon\hpdg{\dgA}\to\dgA$, we get a
well-defined map
\begin{equation*}\label{hdgFM}
\isomap[\dgA,\dgB]\colon\Iso(H^0(\hproj{\dgA\opp\lotimes\dgB}))\tto\hqe{\dgA,\hproj{\dgB}}\qquad\ho{\kke}\mapsto \hq{\dfun{\kke}}\comp\hq{\dgQ_\dgA}^{-1}
\end{equation*}

\begin{prop}\label{prop:exuniq}
For all dg categories $\dgA$ and $\dgB$, the map $\isomap[\dgA,\dgB]$ is bijective.
\end{prop}

\begin{proof}
First of all, we can assume, without loss of generality, that $\dgA$ is h-projective. Given $f\colon\dgA\to\hproj{\dgB}$ in $\Hqe$, by part \eqref{Hqefrac}
of Proposition \ref{Hqemor} there exist dg functors
$\dgI\colon\dgA'\to\dgA$ and $\dgF\colon\dgA'\to\hproj{\dgB}$ with
$\dgI$ a quasi-equivalence such that
$f=\hq{\dgF}\comp\hq{\dgI}^{-1}$. Notice that
$\ext{f}:=\hq{\ext{\dgF}}\comp\hq{\Ind{\dgI}}^{-1}\colon
\hproj{\dgA}\to\hproj{\dgB}$ is such that $\ext{f}\comp\hq{\Yon{\dgA}}=f$ (see the proof of Proposition \ref{prop:Morita}). We denote by
$\kki\in\dgh{{\dgA'}\opp\otimes\dgA}$ and
$\kkf\in\dgh{{\dgA'}\opp\otimes\dgB}$ the objects such that
$\dfun{\kki}=\Yon{\dgA}\comp\dgI$ and $\dfun{\kkf}=\dgF$. Since
$\Ind{\id_{\dgA\opp}\otimes\dgI}\colon\hproj{\dgA\opp\otimes\dgA'}\to
\hproj{\dgA\opp\otimes\dgA}$
is a quasi-equivalence (by Remark \ref{qe} and part \eqref{extresq} of Proposition
\ref{extres}), there exists
$\kkd'\in\hproj{\dgA\opp\otimes\dgA'}$ such that
$\kkd:=\Ind{\id_{\dgA\opp}\otimes\dgI}(\kkd')$ is an h-projective
resolution of $\dgid{\dgA}$. It is easy to see, using Lemma \ref{kerprod}, that
$\Ind{\id_{\dgA\opp}\otimes\dgI}\iso\dgFM{\dgid{\dgA\opp}\otimes\kki}$ and we obtain
\[\kkd\iso\dgFM{\dgid{\dgA\opp}\otimes\kki}(\kkd')\iso
\dgid{\dgA\opp}\otimes_{\dgA}\kkd'\otimes_{\dgA'}\kki\iso
\kkd'\otimes_{\dgA'}\kki,\]
whence $\dgFM{\kkd}\iso\dgFM{\kki}\comp\dgFM{\kkd'}$ by the
associativity of tensor product. Notice that the last isomorphism
above is proved in \cite[Sect.\ 14.6]{Dr}. Taking into account
that $\hq{\dgFM{\kkd}}=\hq{\id_{\hproj{\dgA}}}$, it follows that
$\hq{\dgFM{\kkd'}}=\hq{\dgFM{\kki}}^{-1}=\hq{\Ind{\dgI}}^{-1}$. Setting
$\kke':=\kkd'\otimes_{\dgA'}\kkf\in\dgh{\dgA\opp\otimes\dgB}$, we have
$\dgFM{\kke'}\iso\dgFM{\kkf}\comp\dgFM{\kkd'}$. Therefore
$\hq{\dgFM{\kke'}}=\hq{\dgFM{\kkf}}\comp\hq{\dgFM{\kkd'}}=
\hq{\ext{\dgF}}\comp\hq{\Ind{\dgI}}^{-1}=\ext{f}$.
Taking $\kke\in\hproj{\dgA\opp\otimes\dgB}$ an h-projective resolution
of $\kke'$, this proves that $\hq{\dfun{\kke}}=\hq{\dfun{\kke'}}=\hq{\dgFM{\kke'}}\comp\hq{\Yon{\dgA}}=f$, whence $\isomap[\dgA,\dgB]$ is surjective.

As for injectivity, let $\kke,\kke'\in\hproj{\dgA\opp\otimes\dgB}$ be
such that $\hq{\dfun{\kke}}=\hq{\dfun{\kke'}}$. Then
$\hq{\id_{\dgA\opp}\otimes\dfun{\kke}}=
\hq{\id_{\dgA\opp}\otimes\dfun{\kke'}}$ by Remark \ref{qe}. Hence also
$\hq{\dfun{\dgid{\dgA\opp}}\otimes\dfun{\kke}}=
\hq{\dfun{\dgid{\dgA\opp}}\otimes\dfun{\kke'}}\colon
\dgA\opp\otimes\dgA\to\hproj{\dgA\opp}\otimes\hproj{\dgB}$.
From Lemma \ref{kerprod} we deduce that
$\hq{\dfun{\dgid{\dgA\opp}\otimes\kke}}=
\hq{\dfun{\dgid{\dgA\opp}\otimes\kke'}}$, and so (by Lemma \ref{lem:extension})
\begin{equation}\label{dgFMeq}
\hq{\dgFM{\dgid{\dgA\opp}\otimes\kke}}=
\hq{\dgFM{\dgid{\dgA\opp}\otimes\kke'}}\colon
\hproj{\dgA\opp\otimes\dgA}\to\hproj{\dgA\opp\otimes\dgB}.
\end{equation}
Denoting as before by $\kkd$ an h-projective resolution of
$\dgid{\dgA}$, again by Lemma \ref{kerprod} we have
\[\dgFM{\dgid{\dgA\opp}\otimes\kke}(\kkd)\iso
\dgid{\dgA\opp}\otimes_{\dgA}\kkd\otimes_{\dgA}\kke\iso
\kkd\otimes_{\dgA}\kke.\]
As $\kkd\to\dgid{\dgA}$ is a quasi-isomorphism between dg modules
which are h-flat over $\dgA$, also the induced map
$\kkd\otimes_{\dgA}\kke\to\dgid{\dgA}\otimes_{\dgA}\kke\iso\kke$ is a
quasi-isomorphism, hence a homotopy equivalence, since both the source
and the target are in $\hproj{\dgA\opp\otimes\dgB}$. This proves that
$\ho{\dgFM{\dgid{\dgA\opp}\otimes\kke}(\kkd)}=\ho{\kke}$, and similarly
$\ho{\dgFM{\dgid{\dgA\opp}\otimes\kke'}(\kkd)}=\ho{\kke'}$. As
$\ho{\dgFM{\dgid{\dgA\opp}\otimes\kke}(\kkd)}=
\ho{\dgFM{\dgid{\dgA\opp}\otimes\kke'}(\kkd)}$ by \eqref{dgFMeq}, we
conclude that $\ho{\kke}=\ho{\kke'}$.
\end{proof}

\begin{cor}\label{cor:bij}
If $\dgA$ and $\dgB$ are dg categories, the map of Proposition \ref{prop:Morita} is bijective.
\end{cor}

\begin{proof}
It remains to prove that the map is injective. To this end, we can clearly assume that $\dgB$ is h-projective. Then it is very easy to see that, for every dg category $\dgC$, the map $\isomap[\dgC,\dgB]$ can be identified with the map
\[
\Iso(H^0(\hproj{\dgC\opp\otimes\dgB}))\tto\hqe{\dgC,\hproj{\dgB}}\qquad\ho{\kke}\mapsto \hq{\dfun{\kke}}.
\]
By Proposition \ref{prop:exuniq} this implies that every morphism $\dgC\to\hproj{\dgB}$ in $\Hqe$ can be represented by a dg functor. In particular, this is true when $\dgC=\hproj{\dgA}$.

Therefore, given $f_1,f_2\in\hqe{\hproj{\dgA},\hproj{\dgB}}_c$ such that $f_1\comp\hq{\Yon{\dgA}}=f_2\comp\hq{\Yon{\dgA}}$, there exist two dg functors $\dgF_1,\dgF_2\colon\hproj{\dgA}\to\hproj{\dgB}$ such that $f_i=\hq{\dgF_i}$, for $i=1,2$. Setting
\[
\dgG_i:=\dgF_i\comp\Yon{\dgA}\colon\dgA\to\hproj{\dgB},
\]
for $i=1,2$, we have $\hq{\dgG_1}=\hq{\dgG_2}$ by hypothesis and, by Lemma \ref{lem:extension}, also $\hq{\ext{\dgG_1}}=\hq{\ext{\dgG_2}}$. Thus, in order to conclude that $f_1=f_2$, it is enough to show that $\hq{\dgF_i}=\hq{\ext{\dgG_i}}$, for $i=1,2$.

Indeed, by part \eqref{extresa} of Proposition \ref{extres}, $\ext{\dgG_i}$ (regarded as a dg functor $\dgm{\dgA}\to\dgm{\dgB}$) has a right adjoint $\res{\dgG_i}$. Note that, for all
$M\in\hproj{\dgA}$, we have $\res{\dgG_i}\comp\dgF_i(M)=\hproj{\dgB}(\dgG_i(-),\dgF_i(M))$.
Thus composition with $\dgF_i$ yields a natural map
\[
M=\hproj{\dgA}(\Yon{\dgA}(-),M)\tto\hproj{\dgB}(\dgG_i(-),\dgF_i(M))=\res{\dgG_i}\comp\dgF_i(M),
\]
whence, by adjunction, a natural transformation $\theta\colon\ext{\dgG_i}\to\dgF_i$ with
the property that $H^0(\theta)\rest{\Yon{\dgA}(\dgA)}$ is an isomorphism. Since $H^0(\ext{\dgG_i})$ and $H^0(\dgF_i)$ are continuous (the first by part \eqref{extresi} of Proposition \ref{extres}, the second by assumption),
Remark \ref{rmk:extnat} yields that $\theta$ is a termwise homotopy
equivalence. Hence, by Corollary \ref{rephtp}, we have
$\hq{\dgF_i}=\hq{\ext{\dgG_i}}$, for $i=1,2$.
\end{proof}

\section{The new proof of Theorem \ref{thm:Toen}}\label{sec:Morita}

Let $\dgA$, $\dgB$ and $\dgC$ be dg categories. In view of the basic properties of the derived tensor product, of Remark \ref{qe} and of Proposition \ref{prop:rqr}, we can assume in the proof of Theorem \ref{thm:Toen}, without loss of generality, that these three dg-categories are h-projective. In this way the tensor product does not need to be derived. Putting the results in the previous section together, we get the following maps:
\begin{equation*}\label{eqn:intHoms1}
	\xymatrix{
	\hqe{\dgA\otimes\dgB,\dgC}\ar@{<->}[r]^-{\Psi}_-{1:1}&
	\hqe{\dgA\otimes\dgB,\essim{\dgC}}\ar@{^{(}->}[r]^-{\monomap}&
	\hqe{\dgA\otimes\dgB,\hproj{\dgC}}\ar@{<->}[r]^-{\isomap}_-{1:1}&
	\Iso(H^0(\hproj{(\dgA\otimes\dgB)\opp\otimes\dgC}))
	}
\end{equation*}
where $\Psi$ is induced by the quasi-equivalence $\dgC\hookrightarrow\essim{\dgC}$ while $\monomap:=\monomap[\dgA\otimes\dgB,\essim{\dgC}\hookrightarrow\hproj{\dgC}]$ and $\isomap:=\isomap[\dgA\otimes\dgB,\dgC]$ are the maps with the properties discussed in Propositions \ref{prop:Morita2} and \ref{prop:exuniq}. Obviously, by definition, $\im(\monomap)$ consists of
all $f\in\hqe{\dgA\otimes\dgB,\hproj{\dgC}}$ such that
$\im(H^0(f))\subseteq H^0(\essim{\dgC})$. Using $\isomap$, we get a bijection
between $\im(\monomap)$ and the set of isomorphism classes of the objects
$\kke\in H^0(\hproj{(\dgA\otimes\dgB)\opp\otimes\dgC})$ such that
$H^0(\dfun{\kke})\colon H^0(\dgA\otimes\dgB)\to H^0(\hproj{\dgC})$ factors through $H^0(\essim{\dgC})$.
Thus, by definition, we have a natural bijection between the sets
\begin{equation}\label{eqn:intHoms2}
\xymatrix{
\hqe{\dgA\otimes\dgB,\dgC}\ar@{<->}[rr]^-{1:1}&&\Iso(H^0(\rqr{(\dgA\opp\otimes\dgB\opp)\otimes\dgC})).
}
\end{equation}

On the other hand, we have the following sequence of natural maps of sets:
\begin{equation*}\label{eqn:intHoms3}
\xymatrix{
\hqe{\dgA,\rqr{\dgB\opp\otimes\dgC}}\ar@{^{(}->}[r]^-{\monomap}&
\hqe{\dgA,\hproj{\dgB\opp\otimes\dgC}}\ar@{<->}[r]^-{\isomap}_-{1:1}&
\Iso(H^0(\hproj{\dgA\opp\otimes(\dgB\opp\otimes\dgC)})),
}
\end{equation*}
where $\monomap:=\monomap[\dgA,\rqr{\dgB\opp\otimes\dgC}\hookrightarrow\hproj{\dgB\opp\otimes\dgC}]$ and $\isomap:=\isomap[\dgA,\dgB\opp\otimes\dgC]$ have the properties discussed in Propositions \ref{prop:Morita2} and \ref{prop:exuniq}. In analogy with the previous case,
$\im(\monomap)$ consists of all
$f\in\hqe{\dgA,\hproj{\dgB\opp\otimes\dgC}}$ such that
$\im(H^0(f))\subseteq H^0(\rqr{\dgB\opp\otimes\dgC})$ (here we use that $\rqr{\dgB\opp\otimes\dgC}$ is by definition closed under homotopy equivalences in $\hproj{\dgB\opp\otimes\dgC}$). The map $\isomap$ yields a natural bijection between
$\im(\monomap)$ and the set of isomorphism classes of objects $\kkf\in H^0(\hproj{\dgA\opp\otimes\dgB\opp\otimes\dgC})$ such that
$H^0(\dfun{\kkf})\colon H^0(\dgA)\to H^0(\hproj{\dgB\opp\otimes\dgC})$ factors through $H^0(\rqr{\dgB\opp\otimes\dgC})$. Again by definition, this provides a natural bijection of sets
\begin{equation}\label{eqn:intHoms4}
\xymatrix{
\hqe{\dgA,\rqr{\dgB\opp\otimes\dgC}}\ar@{<->}[rr]^-{1:1}&&\Iso(H^0(\rqr{(\dgA\opp\otimes\dgB\opp)\otimes\dgC})).
}
\end{equation}

If $\dgB$ is the (h-projective) dg category $\K$, then we observed
that $\rqr{\dgB\opp\lotimes\dgC}\iso\essim{\dgC}$. Thus, we get the
natural bijection between the sets $\hqe{\dgA,\dgC}$ and
$\Iso(H^0(\rqr{\dgA\opp\otimes\dgC}))$, which is \eqref{eqn:firstpart}.

As in the statement of Theorem \ref{thm:Toen}, set
$\IHom(\dgB,\dgC):=\rqr{\dgB\opp\lotimes\dgC}$, for two dg categories
$\dgB$ and $\dgC$. If $\dgB$ is an h-projective dg category, we have
$\IHom(\dgB,\dgC)=\rqr{\dgB\opp\otimes\dgC}$ because we do not need to derive the tensor product (see Remark \ref{rmk:hprojdg1}). Due to Proposition \ref{prop:rqr} and the naturality of the bijections
in \eqref{eqn:intHoms2} and \eqref{eqn:intHoms4}, we get a natural
bijection between the sets $\hqe{\dgA\otimes\dgB,\dgC}$ and
$\hqe{\dgA,\IHom(\dgB,\dgC)}$, which is \eqref{eqn:ultima}. So $\Hqe$
is a closed symmetric monoidal category, and this concludes the proof of Theorem \ref{thm:Toen}.

\begin{cor}\label{cor:adj}
Given three dg categories $\dgA$, $\dgB$ and $\dgC$, the dg categories $\IHom(\dgA\lotimes\dgB,\dgC)$ and $\IHom(\dgA,\IHom(\dgB,\dgC))$ are isomorphic in $\Hqe$.
\end{cor}

\begin{proof}
For any dg category $\dgD$ and using \eqref{eqn:ultima} and the associativity of the derived tensor product, we get the following natural bijections:
\[
\xymatrix{
\hqe{\dgD\lotimes(\dgA\lotimes\dgB),\dgC}\ar@{<->}[d]_-{1:1}\ar@{<->}[r]^-{1:1}&\hqe{\dgD\lotimes\dgA,\IHom(\dgB,\dgC)}\ar@{<->}[d]^-{1:1}\\
\hqe{\dgD,\IHom(\dgA\lotimes\dgB,\dgC)}&\hqe{\dgD,\IHom(\dgA,\IHom(\dgB,\dgC))},
}
\]
for every dg category $\dgD$. By Yoneda's lemma, we conclude.
\end{proof}

For two dg categories $\dgA$ and $\dgB$, we denote by $\IHom_c(\hproj{\dgA},\hproj{\dgB})$
the full dg subcategory of $\IHom(\hproj{\dgA},\hproj{\dgB})$
consisting of all $\kkf\in\IHom(\hproj{\dgA},\hproj{\dgB})$ such that
$\ho{\kkf}\in\hqe{\hproj{\dgA},\hproj{\dgB}}_c$ under
\eqref{eqn:firstpart}. We can prove the following, which is Theorem 7.2
in \cite{To} and is usually referred to as derived Morita theory.

\begin{cor}\label{cor:exintHoms2}
Given two dg categories $\dgA$ and $\dgB$, $\IHom(\dgA,\hproj{\dgB})$
and $\hproj{\dgA\opp\lotimes\dgB}$ are isomorphic in $\Hqe$. Moreover, there exist quasi-equivalences
$\IHom_c(\hproj{\dgA},\hproj{\dgB})\to\IHom(\dgA,\hproj{\dgB})$ and $\IHom(\Pe{\dgA},\Pe{\dgB})\to\IHom(\dgA,\Pe{\dgB})$ induced by the Yoneda embedding
$\Yon{\dgA}\colon\dgA\to\Pe{\dgA}\subset\hproj{\dgA}$.
\end{cor}

\begin{proof}
As for the first part of the statement, observe that, in view of Proposition \ref{prop:exuniq}, we have natural bijections
\begin{equation}\label{eqn:b}
\xymatrix{
\hqe{\dgC,\hproj{\dgA\opp\lotimes\dgB}}\ar@{<->}[r]^-{1:1}&\Iso(H^0(\hproj{(\dgC\opp\lotimes\dgA\opp)\lotimes\dgB}))\ar@{<->}[r]^-{1:1}&\hqe{\dgC\lotimes\dgA,\hproj{\dgB}},
	}
\end{equation}
for every dg category $\dgC$. Using \eqref{eqn:ultima}, we get the result by Yoneda's lemma.

For the second part, we argue as at the beginning of the proof of
Theorem 7.2 of \cite{To}. So we have to show that, for
any dg category $\dgC$, the Yoneda embedding $\Yon{\dgA}$ induces a
bijection $\hqe{\hproj{\dgA}\lotimes\dgC,\hproj{\dgB}}'_c\tto
\hqe{\dgA\lotimes\dgC,\hproj{\dgB}}$, where
$\hqe{\hproj{\dgA}\lotimes\dgC,\hproj{\dgB}}'_c$ is the subset of
$\hqe{\hproj{\dgA}\lotimes\dgC,\hproj{\dgB}}$ containing all morphisms
$f$ such that $H^0(f)((\farg,C))$ is  continuous for all $C\in\dgC$.
Indeed, by \eqref{eqn:b}, we get the natural bijection
$\hqe{\dgA\lotimes\dgC,\hproj{\dgB}}\to\hqe{\dgA,\hproj{\dgC\opp\lotimes\dgB}}$.
Similarly, one deduces the natural bijection
$\hqe{\hproj{\dgA}\lotimes\dgC,\hproj{\dgB}}'_c\to
\hqe{\hproj{\dgA},\hproj{\dgC\opp\lotimes\dgB}}_c$. Now
we simply apply Corollary \ref{cor:bij}.

As for perfect dg modules, given a dg category $\dgC$, we consider the
dg functor
$\dgF:=\id_{\dgC\opp}\otimes\Yon{\dgA}\opp\otimes\id_{\dgB}\colon\dgD_1\to\dgD_2$,
where $\dgD_1:=\dgC\opp\lotimes\dgA\opp\lotimes\dgB$ and $\dgD_2:=\dgC\opp\lotimes\Pe{\dgA}\opp\lotimes\dgB$. By \cite[Prop.\ 1.15]{LO}, we have that $\Ind{\dgF}\colon\hproj{\dgD_1}\to\hproj{\dgD_2}$ is a quasi-equivalence. Thus, by Proposition \ref{prop:exuniq} and the same computations as in the proof of Theorem \ref{thm:Toen} above, we get a commutative diagram
\begin{equation*}\label{eqn:intHoms5}
\xymatrix{
\hqe{\dgC\lotimes\dgA,\Pe{\dgB}}\ar@{^{(}->}[r] & \hqe{\dgC\lotimes\dgA,\hproj{\dgB}}\ar@{<->}[r]^-{\isomap}_-{1:1}&
\Iso(H^0(\hproj{\dgD_1}))\ar@{<->}[d]^-{1:1}\\
\hqe{\dgC\lotimes\Pe{\dgA},\Pe{\dgB}}\ar[u]^-{\Yon{\dgC\lotimes\dgA}\comp\farg}\ar@{^{(}->}[r] & \hqe{\dgC\lotimes\Pe{\dgA},\hproj{\dgB}}\ar@{<->}[r]^-{\isomap}_-{1:1}&
\Iso(H^0(\hproj{\dgD_2})),
}
\end{equation*}
where the right vertical bijection is induced by the quasi-equivalence
$\Ind{\dgF}$. Thus the Yoneda embedding induces a natural bijection
between $\hqe{\dgC\lotimes\Pe{\dgA},\Pe{\dgB}}$ and
$\hqe{\dgC\lotimes\dgA,\Pe{\dgB}}$, for all dg categories $\dgC$. By
\eqref{eqn:ultima}, we conclude using Yoneda's lemma.
\end{proof}


\medskip

{\small\noindent{\bf Acknowledgements.}
Part of this paper was written while the second author was visiting the
Department of Mathematics of the Ohio State University and, later,
both authors were at the
Erwin Schr\"odinger International Institute for Mathematical Physics
(Wien) in the context of the thematic program `The Geometry of Topological D-Branes,
Categories, and Applications'. The warm hospitality of these institutions is gratefully acknowledged. We would also like to thank Betrand To\"{e}n for patiently answering our questions and Emanuele Macr\`{i}, Pawel Sosna and Adeel Yusufzai for comments on an early version of this paper. Finally, we would like to express our gratitude to the anonymous reader who pointed out a gap in the proof of Proposition \ref{prop:Morita} contained in the published version of this paper.
}



\begin{thebibliography}{99}

\bibitem{BFN} D.\ Ben--Zvi, J.\ Francis, D.\ Nadler, \emph{Integral transforms and Drinfeld centers in Derived Algebraic Geometry}, J.\ Amer.\ Math.\ Soc.\ {\bf 23} (2010), 909--966.	

\bibitem{BLL} A.\ Bondal, M.\ Larsen, V.\ Lunts, \emph{Grothendieck
ring of pretriangulated categories}, Int.\ Math.\ Res.\ Not.\ {\bf 29}
(2004), 1461--1495.

\bibitem{Br} K.\ S.\ Brown, \emph{Abstract homotopy theory and
generalized sheaf cohomology}, Trans.\ Amer.\ Math.\ Soc.\ {\bf 186} (1973),
419--458.

\bibitem{CS2} A.\ Canonaco, P.\ Stellari, \emph{Non-uniqueness of Fourier-Mukai kernels}, Math.\ Z.\ {\bf 272} (2012), 577--588.

\bibitem{CS1} A.\ Canonaco, P.\ Stellari, \emph{Fourier-Mukai functors: a survey}, EMS Ser.\ Congr.\ Rep., Eur.\ Math.\ Soc.\ (2013), 27--60.

\bibitem{Dr} V.\ Drinfeld, \emph{DG quotients of DG categories},
J.\ Algebra {\bf 272} (2004), 643--691.

\bibitem{Ke1} B.\ Keller, \emph{Deriving DG categories},
Ann.\ Sci.\ \'{E}cole Norm.\ Sup.\ {\bf 27} (1995), 63--102.

\bibitem{Ke} B.\ Keller, \emph{On differential graded categories},
International Congress of Mathematicians, vol.\ II, 151--190, Eur.\
Math.\ Soc.\, Z\"urich (2006).

\bibitem{KL} A.\ Kuznetsov, V.\ Lunts, \emph{Categorical resolutions of irrational singularities}, arXiv:1212.6170.

\bibitem{LO} V.\ Lunts, D.\ Orlov, \emph{Uniqueness of enhancements
for triangulated categories}, J.\ Amer.\ Math.\ Soc.\ {\bf 23} (2010),
853--908.

\bibitem{LS} V.\ Lunts, O.\ M.\ Schn\"urer, \emph{Smoothness of
  equivariant derived categories}, Proc.\ Lond.\ Math.\ Soc.\ {\bf 108} (2014), 1226--1276. 

\bibitem{ML} S.\ Mac Lane, \emph{Homology}, Classics in Mathematics, Springer--Verlag, Berlin (1995).

\bibitem{Mi} B.\ Mitchell, \emph{Rings with several objects},
Adv.\ Math.\ {\bf 8} (1972), 1--161.

\bibitem{N1} A.\ Neeman, \emph{The connection between the K-theory localization theorem of Thomason, Trobaugh and Yao and the smashing subcategories of Bousfield and Ravenel}, Ann.\ Sci.\ \'{E}cole Norm.\ Sup.\ {\bf 25} (1992), 547--566.

\bibitem{Or1} D.\ Orlov, \emph{Equivalences of derived categories and K3 surfaces}, J.\ Math.\ Sci.\ {\bf 84} (1997), 1361--1381.

\bibitem {TaDg} G.\ Tabuada, \emph{Une structure de cat\'{e}gorie de mod\`{e}les de Quillen sur la cat\'{e}gorie des dg-cat\'{e}gories},  C.\ R.\ Math.\ Acad.\ Sci.\ Paris {\bf 340} (2005), 15--19.

\bibitem {Ta} G.\ Tabuada, \emph{Homotopy theory of dg categories via
localizing pairs and Drinfeld's dg quotient}, Homology, Homotopy
Appl.\ {\bf 12}, (2010), 187--219.

\bibitem{To} B.\ To\"en, \emph{The homotopy theory of dg-categories
and derived Morita theory}, Invent.\ Math.\ {\bf 167} (2007), 615--667.

\end{thebibliography}
\end{document}